\documentclass[english,11pt]{amsart}
\pdfoutput=1
\usepackage[T1]{fontenc}
\usepackage[utf8]{inputenc}
\usepackage[dvipsnames]{xcolor}
\usepackage{babel}
\usepackage{verbatim}
\usepackage{prettyref}
\usepackage{amsthm}
\usepackage{amsmath}
\usepackage{amssymb}
\usepackage{mathrsfs}
\usepackage{mathtools}
\usepackage{graphicx}
\usepackage[percent]{overpic}
\usepackage[unicode=true,pdfusetitle,bookmarks=true,bookmarksnumbered=false,bookmarksopen=false,breaklinks=false,pdfborder={0 0 0},backref=false,colorlinks=false]{hyperref}
\usepackage{csquotes}
\usepackage{enumitem}
\usepackage{todonotes}
\usepackage{caption}
\usepackage{subcaption}

\theoremstyle{plain}
\newtheorem{thm}{\protect\theoremname}
\newtheorem{prop}[thm]{\protect\propositionname}
\newtheorem{lem}[thm]{\protect\lemmaname}
\newtheorem{cor}[thm]{\protect\corollaryname}

\theoremstyle{remark}

\theoremstyle{definition}

\numberwithin{thm}{section}

\usepackage{tikz}
\usetikzlibrary{arrows}

\providecommand{\corollaryname}{Corollary}
\providecommand{\definitionname}{Definition}
\providecommand{\examplename}{Example}
\providecommand{\lemmaname}{Lemma}
\providecommand{\notationname}{Notation}
\providecommand{\remarkname}{Remark}
\providecommand{\remarkname}{Definition}
\providecommand{\propositionname}{Proposition}
\providecommand{\theoremname}{Theorem}
\providecommand{\equationname}{Equation}
\providecommand{\figurename}{Figure}
\providecommand{\appendixname}{Appendix}

\providecommand{\sectionname}{Section}
\providecommand{\stepname}{Step}

\newrefformat{def}{\definitionname~\ref{#1}}
\newrefformat{lem}{\lemmaname~\ref{#1}}
\newrefformat{thm}{\theoremname~\ref{#1}}
\newrefformat{cor}{\corollaryname~\ref{#1}}
\newrefformat{prop}{\propositionname~\ref{#1}}
\newrefformat{ex}{\examplename~\ref{#1}}
\newrefformat{sub}{\sectionname~\ref{#1}}
\newrefformat{eq}{\equationname~\eqref{#1}}
\newrefformat{step}{\stepname~\ref*{#1}}
\newrefformat{fig}{\figurename~\ref*{#1}}
\newrefformat{sec}{Section~\ref{#1}}
\newrefformat{app}{\appendixname~\ref{#1}}

\newcommand{\N}{\mathbb N}

\newcommand{\R}{\mathbb R}

\newcommand{\dist}{\operatorname{dist}}

\newcommand{\graph}{\operatorname{graph}}

\newcommand{\Hm}[2]{\mathcal H^{#1}_{#2}}
\newcommand{\HmE}[1]{\mathcal H_{#1}}
\newcommand{\trappedregion}{\mathcal T}
\newcommand{\outwardvariation}[1]{D_{#1}}
\newcommand{\conffact}[1]{u_{#1}}

\renewcommand{\div}{\operatorname{div}}

\newcommand{\deriv}[2]{\frac{d #1}{d #2}}
\newcommand{\partderiv}[2]{\frac{\partial #1}{\partial #2}}

\newcommand{\Cinner}{C_\mathrm{inner}}
\newcommand{\Couter}{C_\mathrm{outer}}

\newcommand{\Router}{R_\mathrm{outer}}
\newcommand{\Rend}{R_\mathrm{end}}

\newcommand{\transportandrescale}[3]{\Pi_{#1, #2}^{#3}}
\newcommand{\Tub}[2]{\operatorname{Tub}(#1, #2)}
\newcommand{\sphere}{S}

\newcommand{\reg}{\operatorname{reg}}
\newcommand{\sing}{\operatorname{sing}}

\newcommand{\solomonwhite}{the maximum principle of Solomon and White \cite[Theorem, p.~686 and Remarks~1 and 2, pp.~690-691]{SolomonWhite89} and \cite[Theorem~4]{White10}%
}

\newcommand{\stepheader}[1]{\vspace{1em}\noindent\textbf{#1}.}
\newcounter{stepnumber}
\newcommand{\step}{\protect\refstepcounter{stepnumber}Step~\thestepnumber}

\title[Horizons diffeomorphic to unit normal bundles]
{Outermost apparent horizons diffeomorphic to unit normal bundles}
\author{Mattias Dahl}
\author{Eric Larsson}

\date{May 21, 2018}

\address{Institutionen f\"or Matematik \\
  Kungliga Tekniska H\"ogskolan \\
  100 44 Stockholm \\
  Sweden} 
\email[Mattias Dahl]{dahl@math.kth.se}
\email[Eric Larsson]{ericlar@math.kth.se}

\begin{document}
\maketitle

\begin{abstract}
Given a submanifold \(S \subset \R^n\) of codimension at least three, we construct an asymptotically Euclidean Riemannian metric on \(\R^n\) with nonnegative scalar curvature for which the outermost apparent horizon is diffeomorphic to the unit normal bundle of \(S\).
\end{abstract}

\setcounter{tocdepth}{1}
\tableofcontents

\makeatletter
\providecommand\@dotsep{5}
\renewcommand{\listoftodos}[1][\@todonotes@todolistname]{%
  \@starttoc{tdo}{#1}}
\makeatother

\section{Introduction}

\label{sec:introduction}

An asymptotically Euclidean manifold is a Riemannian manifold with an end on which the metric approaches the Euclidean metric.
In such a manifold, an outermost apparent horizon is a bounding minimal hypersurface which encloses all other bounding minimal hypersurfaces.
An asymptotically Euclidean manifold can be interpreted as a time symmetric slice of an asymptotically Minkowskian spacetime describing an isolated gravitational system.
The dominant energy condition on the spacetime then means that the time symmetric slice has nonnegative scalar curvature.
By the Hawking--Penrose singularity theorem the outermost apparent horizon must be located inside the event horizon in the spacetime.
The outermost apparent horizon, which can be found with no further data than a time symmetric slice, may therefore serve as a substitute for the event horizon, which depends on the entire spacetime structure.
The question which motivates the work in this paper is the following: Which smooth manifolds can be found as outermost apparent horizons in asymptotically Euclidean manifolds with nonnegative scalar curvature?

\subsection*{Outermost apparent horizons}

We use the convention that the mean curvature of an oriented hypersurface in a Riemannian manifold is the trace of its second fundamental form, with positive sign if a variation in the direction of the oriented normal \(\nu\) increases area. 
In other words, the \(k\)-dimensional sphere of radius \(r\) in \((k+1)\)-dimensional Euclidean space has positive mean curvature \(k/r\) with respect to the outward direction.

For spacetime initial data sets, the concepts of weakly outer trapped surfaces, trapped regions, and outermost apparent horizons are defined using null expansions.
In our setting of Riemannian manifolds these definitions reduce to the following.
Let \((M, g)\) be an \(n\)-dimensional connected Riemannian manifold with an asymptotically Euclidean end.
A compact hypersurface in \(M\) which separates the asymptotically Euclidean end from the rest of the manifold is called \emph{weakly outer trapped} if its mean curvature with respect to the normal directed towards the asymptotically Euclidean end is non-positive.
The \emph{trapped region} \(\trappedregion\) is the union of compact sets with smooth boundary for which the boundary is weakly outer trapped.
The \emph{outermost apparent horizon} of \( (M,g) \) is the boundary \(\partial \trappedregion\) of the trapped region.
Note that the trapped region may be empty.

For asymptotically Euclidean manifolds with nonnegative scalar curvature and smooth outermost apparent horizon, the induced metric on the horizon is conformal to a metric with positive scalar curvature.
The first result in this direction, when the manifold has dimension three, is Hawking's black hole topology theorem \cite[Proposition 9.3.2]{HawkingEllis}.
The general result, in higher dimensions, can be found in work by Cai, Galloway and Schoen \cite{CaiGalloway01}, \cite{GallowaySchoen06}, \cite{Galloway08}.
The proofs use the fact that the horizon is necessarily stable and outer area minimizing.
Hence our motivating question may be refined:
Is the existence of a positive scalar curvature metric sufficient for a compact, bounding manifold to be the outermost apparent horizon in an asymptotically Euclidean manifold of nonnegative scalar curvature?
Note that there are bounding manifolds which do not admit positive scalar curvature metrics, for instance the \(n\)-dimensional torus.
In the present paper we construct many new examples of outermost apparent horizons, but we are far from answering the question completely.

Obvious examples of outermost apparent horizons are spheres, which appear in constant time slices of the Schwarz\-schild spacetime. 
Chru{\'s}ciel and Mazzeo \cite{ChruscielMazzeo03} construct asymptotically Euclidean metrics for which the outermost apparent horizon is diffeomorphic to a number of spheres.
See also work by Corvino \cite{Corvino05}.
In the work \cite{Schwartz08}, Schwartz has a construction of outermost apparent horizons diffeomorphic to a product of spheres. 
Unfortunately, there is a gap in Schwartz's argument where the horizon is implicitly assumed to be connected, and it does not seem to be possible to repair this gap with the methods of that paper.
An inspiration when searching for horizons with nontrivial topology are examples of non-spherical black holes. The first such example is the construction by Emparan and Reall of a black ring spacetime \cite{EmparanReall02}.
Another important example is the ``Black Saturn'' found by Elvang and Figueras \cite{ElvangFiguearas07}.

In this paper we will construct examples of outermost apparent horizons which are diffeomorphic to unit normal bundles of submanifolds of Euclidean space.
Our result is inspired by Schwartz \cite{Schwartz08}, Chru{\'s}ciel and Mazzeo \cite{ChruscielMazzeo03}, as well as by Carr \cite[Theorem~1]{Carr88}.
In \cite{Carr88}, it is proved that boundaries of regular neighborhoods of embedded cell complexes of codimension at least three admit metrics of positive scalar curvature.
For compact submanifolds, the boundary of a regular neighborhood is diffeomorphic to the unit normal bundle.

\subsection*{Statement of results}

Let \(S \subset \R^n\) be a smooth submanifold.
We denote the unit normal bundle of \(S\) by \(UNS\), so that \(UN_xS\) is the sphere of unit vectors in the fiber \(N_x S\) of the normal bundle \(NS\).

The main theorem of this paper is the following.

\begin{thm}\label{thm:main-theorem}
Let \(n \geq 3\) and let \( S \subset \R^n \) be a compact embedded smooth submanifold of dimension \(m\).
Assume that \(n-m \geq 3\). 
For \( \epsilon > 0 \) and \(x \in \R^n \setminus S\), let
\[
\conffact{\epsilon}(x)
\coloneqq
1 + \epsilon^{n - m - 2} \int_S |x - y|^{-(n - 2)} \, dy
\]
and let \(g_\epsilon\) be the Riemannian metric on \( \R^n \setminus S \) defined by
\[
g_\epsilon
\coloneqq
\conffact{\epsilon}^{4/(n - 2)} \delta
\]
where \(\delta\) is the Euclidean metric on \(\R^n\).
For sufficiently small \(\epsilon > 0\) it holds that the outermost apparent horizon \(\Sigma_\epsilon\) of \((\R^n \setminus S, g_\epsilon)\) is diffeomorphic to the unit normal bundle \(UNS\).
In fact, \(\Sigma_\epsilon\) is the graph of a smooth function on \(UNS\) in normal coordinates for \(S\).
\end{thm}

The integrand in the definition of \(\conffact{\epsilon}\) is the Green's function of the Euclidean Laplacian, and hence \(\conffact{\epsilon}\) is harmonic and the scalar curvature of \((\R^n \setminus S, g_\epsilon)\) is zero (see \cite[Theorem~1.159]{Besse}).
For large \(x\) we have
\[
\conffact{\epsilon}(x)
= 1 + \epsilon^{n-m-2} \HmE{m} (S) |x|^{-(n-2)} + O(|x|^{-(n-1)}),
\]
where \(\HmE{m}\) denotes the \(m\)-dimensional Hausdorff measure with respect to the Euclidean metric \(\delta\).
Thus \((\R^n \setminus S, g_\epsilon)\) has an asymptotically Euclidean end.

A neighborhood of \(S\) is another end of the manifold \((\R^n \setminus S, g_\epsilon)\), and the function \(\conffact{\epsilon}\) tends to infinity at \(S\).
The proof will show that the outermost apparent horizon \(\Sigma_\epsilon\) encloses a tubular hypersurface of Euclidean radius proportional to \(\epsilon\) around \(S\). 
It is possible to modify the harmonic function \(\conffact{\epsilon}\) inside this tubular hypersurface to give a smooth superharmonic function on all of \(\R^n\), thus removing the end at \(S\).
The modification of \(\conffact{\epsilon}\) can be done by appropriately cutting off the singularity at zero of the function \( r \mapsto |r|^{-(n-2)} \) appearing in the integral in the definition of \(u_\epsilon\).
This modification will not change the location of the outermost apparent horizon, and we get the same type of examples of horizons in asymptotically flat manifolds but with non-negative scalar curvature and no end near \(S\).
We conclude the following.

\begin{thm}\label{thm:main-corollary}
Let \( S \subset \R^n \) be a compact embedded smooth submanifold of codimension at least 3. 
Then there is an asymptotically Euclidean metric on \(\R^n\) with non-negative scalar curvature for which the outermost apparent horizon is diffeomorphic to the unit normal bundle \(UNS\).
In fact, the horizon is the graph of a smooth function on \(UNS\) in normal coordinates for \(S\).
\end{thm}

The basic idea of the proof of \prettyref{thm:main-theorem} is this:
Since the conformal factor is close to one outside of a neighborhood of \(S\) which shrinks as \(\epsilon\) tends to zero, the horizon (if it exists) is expected to approach \(S\) in this limit.
Close to \(S\), the difference between \(S\) and its tangent space should be negligible, so the horizon \(\Sigma_\epsilon\) should be close to the one we would have if \(S\) were a linear subspace, in which case \(\Sigma_\epsilon\) would be a cylinder.
If the horizon is locally close to a cylinder around the tangent space of \(S\), then it should be globally diffeomorphic to the unit normal bundle of \(S\).

We now present some examples of outermost apparent horizons that can be found with \prettyref{thm:main-theorem}.

\begin{itemize}
\item
Horizons diffeomorphic to products of spheres:
Let \(n\) and \(m\) be positive integers such that \(n \geq m + 3\).
Let \(S = S^m\) be the unit sphere in \(\R^{m + 1} \subset \R^n\).
Then \prettyref{thm:main-theorem} gives an outermost apparent horizon diffeomorphic to \(S^m \times S^{n - m - 1}\).
\item
Horizons can have many components:
Let \(n \geq 3\).
Let \(S\) be a set of \(k\) points in \(\R^n\).
Then \prettyref{thm:main-theorem} gives an outermost apparent horizon \(\Sigma\) diffeomorphic to the disjoint union of \(k\) spheres of dimension \((n-1)\).
In general, we can find examples of disconnected horizons using any choice of a disconnected submanifold \(S\).
Refining the proof of \prettyref{thm:main-theorem} one can probably allow the components of \(S\) to have different dimensions, which would allow examples like the ``Black Saturn''.   
\item
Horizons can have any fundamental group:
Let \(\pi\) be a finitely presented group.
Using the generators and relations of \(\pi\), it is not difficult to construct a compact hypersurface \(S \subset \R^5\) such that \(\pi_1(S) = \pi\).
Embed \(\R^5\) in \(\R^7\).
\prettyref{thm:main-theorem} gives us an apparent horizon which is diffeomorphic to the unit normal bundle of \(S \subset \R^7\).
The fiber of this bundle is \(S^2\) so the long exact sequence for the homotopy groups of a fibration tells us that \(\pi_1(UNS) = \pi_1(S) = \pi\).
This shows that every finitely presented group is the fundamental group of an outermost apparent horizon in an asymptotically flat \(7\)-dimensional scalar flat Riemannian manifold.
\end{itemize}

\subsection*{Overview of the paper}

In \prettyref{sec:rescaled-metrics-and-their-convergence} we prove that the function \(\conffact{\epsilon}\) near a point of \(S\) is close to the function \(\conffact{\infty}\) obtained by replacing \(S\) by its tangent space at the point.
This is proved by explicit computations involving rescaled versions of the function \(\conffact{\epsilon}\).

In \prettyref{sec:the-mean-curvature-of-tubular-hypersurfaces} we determine the mean curvature in \(g_\epsilon\) of tubular hypersurfaces around \(S\).
We find constants \(\Cinner\), \(\Couter\) and \(\Router\) such that the mean curvature is negative for tubular hypersurfaces with radius smaller than \(\Cinner \epsilon\) and positive for tubular hypersurfaces with radius between \(\Couter \epsilon\) and \(\Router\).

In \prettyref{sec:the-location-of-outer-area-minimizing-stationary-hypersurfaces} we use these tubular hypersurfaces, the maximum principle of Solomon and White, and a convergence argument to determine the location of outer area minimizing stationary hypersurfaces.
The conclusion is that any such hypersurface must be located between the tubular hypersurfaces of radii \(\Cinner \epsilon\) and \(\Couter \epsilon\).

In \prettyref{sec:tubularity-of-outer-area-minimizing-stationary-hypersurfaces} we apply a convergence argument to prove that the outer area minimizing stationary hypersurfaces are graphs of smooth functions on \(UNS\) in normal coordinates for \(S\).

Finally, in \prettyref{sec:proof-of-the-main-theorem}, we combine the previous results and prove that there is a unique outer area minimizing stationary hypersurface.
This hypersurface is then shown to coincide with the outermost apparent horizon, which proves \prettyref{thm:main-theorem}.

\subsection*{Acknowledgements}

We want to thank
John Andersson,
Lars Andersson,
Alessandro Carlotto,
Michael Eichmair,
Christos Mantoulidis,
Anna Sakovich, and
Fernando Schwartz
for helpful comments and discussions related to the work in this paper.

The many insightful remarks we received from an anonymous referee has helped us improve the paper in ways we would not have been able to by ourselves.
We are deeply thankful for the work and effort made by this person.

Much of this work was done when the authors visited the IHP in Paris for the fall 2015 program on mathematical general relativity. We want to express our gratitude to the institute for its hospitality.

\section{Rescaled metrics and their convergence}

\label{sec:rescaled-metrics-and-their-convergence}

\newcommand{\xipar}{\xi^\parallel}
\newcommand{\xikpar}{\xi_k^\parallel}
\newcommand{\etapar}{\eta^\parallel}
\newcommand{\hatetapar}{\hat{\eta}^\parallel}

The purpose of this section is to prove that, after suitable rescaling and change of coordinates, the metrics \(g_\epsilon\) converge as \(\epsilon\) tends to zero.
This is done in \prettyref{cor:metric-convergence}.
The computations needed are mainly contained in \prettyref{lem:integral-convergence-and-bound}, where rescalings of the integral in the definition of \(\conffact{\epsilon}\) are studied.

For \((x_\infty, x, \epsilon) \in \R^n \times \R^n \times \R^+\) define the map \(\transportandrescale{x_\infty}{x}{\epsilon} \colon T_{x_\infty}\R^n \to \R^n\) to be the composition \(T_{x_\infty}\R^n \to T_x\R^n \to \R^n\) of the isomorphism \(T_{x_\infty}\R^n \to T_x\R^n\) of tangent spaces given by parallel transport with respect to the Euclidean metric \(\delta\), and the map \(\zeta \mapsto \exp^{\delta}(\epsilon \zeta)\).
Here \(\exp^\delta\) is the exponential map for the Euclidean metric \(\delta\).
This map is used to focus attention on an \(\epsilon\)-neighborhood of the point \(x\); the pullback by \(\transportandrescale{x_\infty}{x}{\epsilon}\) of the ball of radius \(\epsilon\) around \(x\) is the ball of radius \(1\) around the origin in \(T_{x_\infty}\R^n\).

In the following it is helpful to keep in mind that \(\transportandrescale{x_\infty}{x}{\epsilon}(\zeta)\) is the same as \(x + \epsilon \zeta\) after natural identifications of points and vectors.

\begin{lem}\label{lem:integral-convergence-and-bound}
Fix \(\gamma > m\) and \(\beta_1, \beta_2 > 0\). 
Let \(x_k \in S\) and \(\epsilon_k > 0\) be sequences with \(x_k \to x_\infty\) and \(\epsilon_k \to 0\).

\begin{enumerate}[label=\emph{\textbf{(\roman*)}},ref=(\roman*)]
\item \label{enum:integral-bound}
There are constants \(R\) and \(C\) depending on \((S, x_\infty, \gamma, \beta_1, \beta_2)\) such that if \(k\) is such that \(|x_k - x_\infty| + \epsilon_k \beta_1 \leq R/2\), and \(\zeta \in T_{x_\infty}\R^n\) satisfies \(|\zeta| \leq \beta_1\) and \(\dist^\delta(\zeta, T_{x_\infty} S) \geq \beta_2\) then
\[
\epsilon_k^{\gamma - m} \int_S |\transportandrescale{x_\infty}{x_k}{\epsilon_k}(\zeta) - y|^{-\gamma} \, dy
<
C.
\]
\item \label{enum:integral-convergence}
For all \(\zeta \notin T_{x_\infty} S\) it holds that
\[
\lim_{k \to \infty} \epsilon_k^{\gamma - m} \int_S |\transportandrescale{x_\infty}{x_k}{\epsilon_k}(\zeta) - y|^{-\gamma} \, dy
=
\int_{T_{x_\infty} S} |\zeta - \eta|^{-\gamma} \, d\eta.
\]
\end{enumerate}
In other words, the sequence of functions
\[
F_k(\zeta) \coloneqq
\epsilon_k^{\gamma - m} \int_S |\transportandrescale{x_\infty}{x_k}{\epsilon_k}(\zeta) - y|^{-\gamma} \, dy
\]
is uniformly bounded for \(\zeta \in T_{x_\infty}\R^n\) with \(|\zeta| \leq \beta_1\) and \(\dist^\delta(\zeta, T_{x_\infty} S) \geq \beta_2\), and it converges pointwise to the function
\[
F_\infty(\zeta)
\coloneqq
\int_{T_{x_\infty}S} |\zeta - \eta|^{-\gamma} \, d\eta
\]
on \(T_{x_\infty} \R^n \setminus T_{x_\infty} S\).
\end{lem}
\begin{proof}
Let \(\pi \colon T_{x_\infty} \R^n \to T_{x_\infty} S\) denote the orthogonal projection and recall that \(N_{x_\infty}S\) denotes the linear subspace of \(T_{x_\infty} \R^n\) orthogonal to \(T_{x_\infty}S\).
Choose \(R > 0\) such that there is an open neighborhood \(A \subset T_{x_\infty} S\) of \(0\), and a smooth function \(\sigma \colon A \to N_{x_\infty} S\) with
\[
S \cap B_R(x_\infty)
=
\exp^\delta(\graph \sigma)
\]
and such that \(\sigma\) can be extended smoothly to a neighborhood of \(\overline A\), with Lipschitz constant smaller than \(\beta_2/(4\beta_1)\).
We will consider integrals over \(S \setminus B_R(x_\infty)\) and \(S \cap B_R(x_\infty)\) separately.

\stepheader{Part \ref{enum:integral-bound}: Uniform boundedness}
Note that
\[\begin{split}
|\transportandrescale{x_\infty}{x_k}{\epsilon_k}(\zeta) - x_\infty|
&\leq
|x_k - x_\infty| + \epsilon_k |\zeta| \\
&\leq
|x_k - x_\infty| + \epsilon_k \beta_1.
\end{split}\]
Let \(k\) be such that \(|x_k - x_\infty| + \epsilon_k \beta_1 \leq R/2\).
Then
\[\begin{split}
|\transportandrescale{x_\infty}{x_k}{\epsilon_k}(\zeta) - y| 
&\geq
|y - x_\infty| - |\transportandrescale{x_\infty}{x_k}{\epsilon_k}(\zeta) - x_{\infty}|
\\
&\geq 
R - R/2 
\\
&= 
R/2
\end{split}
\]
for \(y \in S \setminus B_R(x_\infty) \), so 
\[\begin{split}
\int_{S \setminus B_R(x_\infty)} |\transportandrescale{x_\infty}{x_k}{\epsilon_k}(\zeta) - y|^{-\gamma} \, dy
&\leq
\int_{S \setminus B_R(x_\infty)} (R/2)^{-\gamma} \, dy \\
&\leq
(R/2)^{-\gamma} \HmE{m}(S).
\end{split}\]
Therefore
\begin{equation}\label{eq:integral-tail}
\epsilon_k^{\gamma - m} \int_{S \setminus B_R(x_\infty)} |\transportandrescale{x_\infty}{x_k}{\epsilon_k}(\zeta) - y|^{-\gamma} \, dy \to 0
\end{equation}
as \(k \to \infty\), since \(\gamma > m\) and \(\epsilon_k \to 0\).

We parametrize \(S \cap B_R(x_\infty)\) by \(\hatetapar \mapsto \exp^\delta(\hatetapar + \sigma(\hatetapar))\) where \(\hatetapar \in A\).
The Jacobian of this map is represented by the matrix
\[\begin{bmatrix}
I\\
J\sigma(\hatetapar)
\end{bmatrix}.\]
With this parametrization, the integral
\[
\epsilon_k^{\gamma - m} \int_{S \cap B_R(x_\infty)} |\transportandrescale{x_\infty}{x_k}{\epsilon_k}(\zeta) - y|^{-\gamma} \, dy
\]
can be expressed as
\[
\epsilon_k^{\gamma - m} \int_A \left|\transportandrescale{x_\infty}{x_k}{\epsilon_k}(\zeta) - \exp^\delta\left(\hatetapar + \sigma(\hatetapar)\right)\right|^{-\gamma} \mathcal J(\hatetapar)\, d\hatetapar
\]
where
\[
\mathcal J(\hatetapar)
=
\sqrt{\det\left(\begin{bmatrix}
I\\
J\sigma(\hatetapar)
\end{bmatrix}^T
\begin{bmatrix}
I\\
J\sigma(\hatetapar)
\end{bmatrix}\right)}
=
\sqrt{\det\left(I + J\sigma(\hatetapar)^T J\sigma(\hatetapar) \right)}.
\]
If \(x_k \in S \cap B_R(x_\infty)\) then \(x_k = \exp^\delta(\xikpar + \sigma(\xikpar))\) for some \(\xikpar \in A\).
It holds that
\[
\transportandrescale{x_\infty}{x_k}{\epsilon_k}(\zeta)
=
\exp^\delta\left(\xikpar + \sigma(\xikpar) + \epsilon_k \zeta\right),
\]
so
\[ \begin{split}
&\epsilon_k^{\gamma - m} \int_{S \cap B_R(x_\infty)} |\transportandrescale{x_\infty}{x_k}{\epsilon_k}(\zeta) - y|^{-\gamma} \, dy \\
&\qquad =
\epsilon_k^{\gamma - m} \int_A \left|\exp^\delta\left(\xikpar + \sigma(\xikpar) + \epsilon_k \zeta\right) - \exp^\delta\left(\hatetapar + \sigma(\hatetapar)\right)\right|^{-\gamma} \mathcal J(\hatetapar)\, d\hatetapar \\
&\qquad =
\epsilon_k^{\gamma - m} \int_A \left|\xikpar + \sigma(\xikpar) + \epsilon_k \zeta - \hatetapar + \sigma(\hatetapar)\right|^{-\gamma} \mathcal J(\hatetapar)\, d\hatetapar \\
&\qquad =
\epsilon_k^{-m} \int_A \left|\zeta - \frac{\hatetapar - \xikpar}{\epsilon_k} - \frac{\sigma(\hatetapar) - \sigma(\xikpar)}{\epsilon_k}\right|^{-\gamma} \mathcal J(\hatetapar)\, d\hatetapar.
\end{split} \]
By a change of variables 
\(\hatetapar = \xikpar + \epsilon_k \etapar\)
we see that
\begin{equation}\label{eq:rescaled-parametrization}
\begin{split}
&\epsilon_k^{\gamma - m} \int_{S \cap B_R(x_\infty)} |\transportandrescale{x_\infty}{x_k}{\epsilon_k}(\zeta) - y|^{-\gamma} \, dy \\
&\qquad=
\int_{(A - \xikpar)/\epsilon_k} \left|\zeta - \etapar - \frac{\sigma(\xikpar + \epsilon_k \etapar) - \sigma(\xikpar)}{\epsilon_k}\right|^{-\gamma} \mathcal J(\xikpar + \epsilon_k \etapar)\, d\etapar.
\end{split}
\end{equation}
We will now bound the integrand by an integrable function of \(\etapar\).
Let \(L\) denote the Lipschitz constant of \(\sigma\) and recall that
\[
L < \beta_2/(4\beta_1).
\]
Then
\begin{equation}\label{eq:the-inequality-involving-L}
\left| \frac{\sigma(\xikpar + \epsilon_k \etapar) - \sigma(\xikpar)}{\epsilon_k} \right|
\leq
\frac{L \epsilon_k |\etapar|}{\epsilon_k}
=
L |\etapar|
<
\frac{\beta_2}{4 \beta_1} |\etapar|.
\end{equation}
By orthogonal decomposition of \(\zeta \in T_{x_\infty} S \oplus N_{x_\infty} S\) we get
\begin{equation}\label{eq:orthogonal-decomposition}
\begin{split}
&\left| \zeta - \etapar - \frac{\sigma(\xikpar + \epsilon_k \etapar) - \sigma(\xikpar)}{\epsilon_k} \right|^2
\\
&\qquad=
|\pi(\zeta) - \etapar|^2
+
\left| (\zeta - \pi(\zeta)) - \frac{\sigma(\xikpar + \epsilon_k \etapar) - \sigma(\xikpar)}{\epsilon_k} \right|^2.
\end{split}
\end{equation}
By the triangle inequality and \eqref{eq:the-inequality-involving-L} we further have
\[\begin{split}
\left| (\zeta - \pi(\zeta)) - \frac{\sigma(\xikpar + \epsilon_k \etapar) - \sigma(\xikpar)}{\epsilon_k} \right|
&\geq
|\zeta - \pi(\zeta)| - \left| \frac{\sigma(\xikpar + \epsilon_k \etapar) - \sigma(\xikpar)}{\epsilon_k} \right|
\\
&\geq
\dist^\delta(\zeta, T_{x_\infty} S) - \frac{\beta_2}{4 \beta_1} |\etapar|
\\
&\geq
\beta_2 - \frac{\beta_2}{4 \beta_1} |\etapar|
\end{split}\]
so if \(|\etapar| < 2\beta_1\) then
\[
\left| (\zeta - \pi(\zeta)) - \frac{\sigma(\xikpar + \epsilon_k \etapar) - \sigma(\xikpar)}{\epsilon_k} \right|
>
\frac{\beta_2}{2}.
\]
This gives a good bound for small \(|\etapar|\).
For large \(|\etapar|\) we will use the following bound instead.
The orthogonal decomposition \eqref{eq:orthogonal-decomposition} tells us that
\[\begin{split}
\left| \zeta - \etapar - \frac{\sigma(\xikpar + \epsilon_k \etapar) - \sigma(\xikpar)}{\epsilon_k} \right|
&\geq
|\pi(\zeta) - \etapar|
\\
&\geq
|\etapar| - |\pi(\zeta)|
\\
&\geq
|\etapar| - |\zeta|
\\
&\geq
|\etapar| - \beta_1.
\end{split}\]
Finally note that \(\mathcal J(\hatetapar)\) is bounded by some constant \(C' = C'(S, x_\infty, \beta_1, \beta_2)\) on \(A\) since \(\sigma\) can be extended smoothly beyond \(\overline A\).
Hence
\[
\left| \frac{\zeta - \etapar - (\sigma(\xikpar + \epsilon_k \etapar) - \sigma(\xikpar))}{\epsilon_k} \right|^{-\gamma} \mathcal J(\xikpar + \epsilon_k \etapar)
\]
is bounded from above by the function
\[
\Phi(\etapar)
\coloneqq
\begin{cases}
C'\left(\frac{\beta_2}{2}\right)^{-\gamma} & \text{ if } |\etapar| < 2\beta_1,\\
C'(|\etapar| - \beta_1)^{-\gamma} & \text{ if } |\etapar| \geq 2\beta_1,
\end{cases}
\]
which is integrable on all of \(T_{x_\infty} S\) since \(\gamma > m\).
This means that we have now shown that
\[\begin{split}
&\epsilon_k^{\gamma - m} \int_S |\transportandrescale{x_\infty}{x_k}{\epsilon_k}(\zeta) - y|^{-\gamma} \, dy
\\
&\qquad
\leq
(R/2)^{-m} \HmE{m}(S) \\
&\qquad \qquad
+
\int_{(A - \xikpar)/\epsilon_k} \left|\zeta - \etapar - \frac{\sigma(\xikpar + \epsilon_k \etapar) - \sigma(\xikpar)}{\epsilon_k} \right|^{-\gamma} \mathcal J(\xikpar + \epsilon_k \etapar)\, d\etapar
\\
&\qquad
\leq
(R/2)^{-m} \HmE{m}(S)
+
\int_{(A - \xikpar)/\epsilon_k} \Phi(\etapar) \, d\etapar
\\
&\qquad
\leq
(R/2)^{-m} \HmE{m}(S)
+
\int_{T_{x_\infty} S} \Phi(\etapar) \, d\etapar
\end{split}\]
for \(x_k \in \R^n\) and \(\epsilon_k > 0\) such that \(|x_k - x_\infty| + \epsilon_k \beta_1 \leq R/2\).
Choosing
\[
C(S, x_\infty, \gamma, \beta_1, \beta_2)
\coloneqq
(R/2)^{-m} \HmE{m}(S) + \int_{T_{x_\infty} S} \Phi(\etapar) \, d\etapar
\]
completes the proof of the boundedness statement in part \ref{enum:integral-bound}.

\stepheader{Part \ref{enum:integral-convergence}: Pointwise convergence}
We will now apply the dominated convergence theorem to pass to the limit in \eqref{eq:rescaled-parametrization}.
As preparation we first note that
\[
\lim_{k \to \infty} 
\frac{\sigma(\xikpar + \epsilon_k \etapar) - \sigma(\xikpar)}{\epsilon_k}
=
\lim_{\substack{\xikpar \to x_\infty\\ \epsilon_k \to 0}} 
\frac{\sigma(\xikpar + \epsilon_k \etapar) - \sigma(\xikpar)}{\epsilon_k}
= 0
\]
since \(J\sigma(x_\infty) = 0\).
For the same reason,
\[
\lim_{k \to \infty} \mathcal J(\xikpar + \epsilon_k \etapar)
= \lim_{\substack{\xikpar \to x_\infty\\ \epsilon_k \to 0}} 
\sqrt{\det(I + J\sigma(\xikpar + \epsilon_k \etapar)^T J\sigma(\xikpar + \epsilon_k \etapar))}
= 1.
\]
Let \(\chi_{(A - \xikpar)/\epsilon_k}\) denote the characteristic function of the set \((A - \xikpar)/\epsilon_k\). 
Using the existence of the integrable function \(\Phi\) we can apply the dominated convergence theorem to see that
\[ \begin{split}
&\lim_{k \to \infty} 
\int_{(A - \xikpar)/\epsilon_k} 
\left| \zeta - \etapar - \frac{\sigma(\xikpar + \epsilon_k \etapar) - \sigma(\xikpar)}{\epsilon_k} \right|^{-\gamma} 
\mathcal J(\xikpar + \epsilon_k \etapar) \, d\etapar \\
&\qquad =
\lim_{k \to \infty} 
\int_{T_{x_\infty}S} 
\chi_{\frac{A - \xikpar}{\epsilon_k}}(\etapar) 
\left| \zeta - \etapar - \frac{\sigma(\xikpar + \epsilon_k \etapar) - \sigma(\xikpar)}{\epsilon_k} \right|^{-\gamma} \\
&\qquad\qquad\qquad\qquad\qquad\qquad\qquad\qquad\qquad\qquad\qquad\qquad \cdot \mathcal J(\xikpar + \epsilon_k \etapar)\, d\etapar \\
&\qquad =
\int_{T_{x_\infty}S} 
\lim_{k \to \infty} 
\left| \zeta - \etapar - \frac{\sigma(\xikpar + \epsilon_k \etapar) - \sigma(\xikpar)}{\epsilon_k} \right|^{-\gamma} 
\mathcal J(\xikpar + \epsilon_k \etapar)\, d\etapar \\
&\qquad =
\int_{T_{x_\infty}S} |\zeta - \etapar|^{-\gamma} \, d\etapar.
\end{split} \]
This together with \eqref{eq:rescaled-parametrization} means that we have now shown that
\[
\lim_{k \to \infty} \epsilon_k^{\gamma - m} 
\int_{S \cap B_R(x_\infty)} |\transportandrescale{x_\infty}{x_k}{\epsilon_k}(\zeta) - y|^{-\gamma} \, dy
=
\int_{T_{x_\infty}S} |\zeta - \etapar|^{-\gamma} \, d\etapar,
\]
which together with \eqref{eq:integral-tail} completes the proof of part \ref{enum:integral-convergence}.
\end{proof}

In the next proposition we prove that the functions \(F_k\) from \prettyref{lem:integral-convergence-and-bound} converge in \(C^r(K)\) for compact domains \(K\).

\begin{prop}\label{prop:uniform-integral-convergence}
Fix \(\gamma > m\). 
Let \(x_k \in S\) and \(\epsilon_k > 0\) be sequences with \(x_k \to x_\infty\) and \(\epsilon_k \to 0\).
Let \(K \subset T_{x_\infty} \R^n \setminus T_{x_\infty}S\) be compact.
Then for \(r \in \N\) the functions \(F_k\) converge to \(F_\infty\) in \(C^r(K)\)  as \(k \to \infty\).
\end{prop}
\begin{proof}
Fix \(r \in \N\).
We will apply \prettyref{lem:integral-convergence-and-bound}.
Let
\[
\beta_1
\coloneqq
\max_{\zeta \in K} |\zeta|,
\]
\[
\beta_2
\coloneqq
\min_{\zeta \in K} \dist^\delta(\zeta, T_{x_\infty} S).
\]
For each \(l \in \N\) let
\[
D_l \coloneqq \prod_{i = 0}^{l-1} (\gamma + i).
\]
We have
\[
\left| \partderiv{}{x^{i_1}} \cdots \partderiv{}{x^{i_l}} |x - y|^{-\gamma} \right|
\leq
D_l |x - y|^{-(\gamma + l)}
\]
so
\[ \begin{split}
&\left| \partderiv{}{\zeta^{i_1}} \cdots \partderiv{}{\zeta^{i_l}} \epsilon_k^{\gamma - m} |\transportandrescale{x_\infty}{x_k}{\epsilon_k}(\zeta) - y|^{-\gamma} \right| \\
&\qquad \qquad \qquad \qquad \qquad \qquad
\leq
D_l \epsilon_k^{\gamma + l - m} |\transportandrescale{x_\infty}{x_k}{\epsilon_k}(\zeta) - y|^{-(\gamma + l)} .
\end{split}\]
By applying \prettyref{lem:integral-convergence-and-bound} for all exponents \(\{\gamma + l \mid 0 \leq l \leq r+1\}\) we get constants \(C_l\) such that
\[
D_l \epsilon_k^{\gamma + l - m}  \int_S |\transportandrescale{x_\infty}{x_k}{\epsilon_k}(\zeta) - y|^{-(\gamma + l)} \, dy
\leq
C_l.
\]
This means that
\[\begin{split}
\partderiv{}{\zeta^{i_1}} \cdots \partderiv{}{\zeta^{i_l}} F_k(\zeta)
&=
\partderiv{}{\zeta^{i_1}} \cdots \partderiv{}{\zeta^{i_l}} \epsilon_k^{\gamma - m} \int_S |\transportandrescale{x_\infty}{x_k}{\epsilon_k}(\zeta) - y|^{-\gamma} \, dy
\\
&=
\epsilon_k^{\gamma - m} \int_S \partderiv{}{\zeta^{i_1}} \cdots \partderiv{}{\zeta^{i_l}} |\transportandrescale{x_\infty}{x_k}{\epsilon_k}(\zeta) - y|^{-\gamma} \, dy
\\
&\leq
D_l \epsilon_k^{\gamma + l - m} \int_S |\transportandrescale{x_\infty}{x_k}{\epsilon_k}(\zeta) - y|^{-(\gamma + l)} \, dy
\\
&\leq
C_l.
\end{split}\]
Hence we have uniform bounds for \(||F_k||_{C^{r+1}(K)}\).

From the uniform bounds on \(||F_k||_{C^{r+1}(K)}\) and the pointwise convergence of \(F_k\) to \(F_\infty\) it follows that \(F_k \to F_\infty\) in \(C^r(K)\).
\end{proof}

As a corollary we can determine what the metrics \(g_{\epsilon_k}\) converge to after appropriate rescaling and change of coordinates.

\begin{cor}\label{cor:metric-convergence}
Let \(K \subset T_{x_\infty} \R^n \setminus T_{x_\infty}S\) be compact.
Let \(x_k \in S\) and \(\epsilon_k > 0\) be sequences with \(x_k \to x_\infty\) and \(\epsilon_k \to 0\).
Then the rescaled pullback metrics \(\epsilon_k^{-2}(\transportandrescale{x_\infty}{x_k}{\epsilon_k})^*(g_{\epsilon_k})\) on \(K\) converge in every \(C^r(K)\) to
\[
\conffact{\infty}^{4/(n - 2)} \delta
\]
where \(\delta\) is the Euclidean metric on \(K \subset  T_{x_\infty}\R^n\) and
\[
\conffact{\infty}(\zeta)
\coloneqq
1 + \int_{T_{x_\infty}S} |\zeta - \eta|^{-(n - 2)} \, d\eta.
\]
\end{cor}
\begin{proof}
With \(\gamma = n-2\) we have \(\conffact{\epsilon_k}(\transportandrescale{x_\infty}{x_k}{\epsilon_k}(\zeta)) = 1 + F_k(\zeta)\) and \(\conffact{\infty}(\zeta) = 1 + F_\infty(\zeta)\).
Since \(n - 2 > m\) it follows from \prettyref{prop:uniform-integral-convergence} that
\[
\conffact{\epsilon_k}(\transportandrescale{x_\infty}{x_k}{\epsilon_k}(\zeta)) \to \conffact{\infty}(\zeta)
\]
in \(C^r(K)\) as \(k \to \infty\).
The corollary follows since \({\epsilon_k^{-2} (\transportandrescale{x_\infty}{x_k}{\epsilon_k})^*(\delta) = \delta}\).
\end{proof}

\section{The mean curvature of tubular hypersurfaces}

\label{sec:the-mean-curvature-of-tubular-hypersurfaces}

By \emph{the tubular hypersurface of radius \(a\) around \(S\)}, denoted \(\Tub{S}{a}\), we mean the image of \(UNS\) under the map \(\zeta \mapsto \exp^\delta(a \zeta)\).
Recall that \(\exp^\delta\) is the exponential map for the Euclidean metric \(\delta\).

We need to understand the behaviour of \(S\) and \(\Tub{S}{a}\) under the rescaling studied in the previous section.
By writing \(S\) as a graph over \(T_{x_\infty} S\) in normal coordinates, we conclude the following.

\begin{lem} \label{lem:convergence-of-S}
Let \(x_k \in S\) and \(\epsilon_k > 0\) be sequences with \(x_k \to x_\infty\) and \(\epsilon_k \to 0\).
Then 
\[
(\transportandrescale{x_\infty}{x_k}{\epsilon_k})^{-1}(S) \to T_{x_\infty} S
\] 
and 
\[
(\transportandrescale{x_\infty}{x_k}{\epsilon_k})^{-1}(\Tub{S}{a\epsilon_k}) \to \Tub{T_{x_\infty} S}{a}
\]
smoothly on compact subsets of \(T_{x_\infty}\R^n\) as \(k \to \infty\).
\end{lem}
Here, smooth convergence of submanifolds \(N_k\) means that if \(k\) is sufficiently large then in normal coordinates for the limit submanifold \(N_\infty\), each \(N_k\) is the graph of a smooth function on \(N_\infty\), and these functions converge smoothly to \(0\) as \(k \to \infty\).
 
The cylinder \(\Tub{T_{x_\infty} S}{a}\) has mean curvature \((n-m-1)a^{-1}\) in the metric \(\delta\).
Thus, we can use \prettyref{lem:convergence-of-S} to compute the mean curvature of \(\Tub{S}{a}\) in the Euclidean metric \(\delta\) when the radius \(a\) is small.
See also for instance \cite[Section~4]{MahmoudiMazzeoPacard06}.

\begin{lem}\label{lem:mean-curvature-in-delta}
The mean curvature of \(\Tub{S}{a}\) in the Euclidean metric \(\delta\) is
\[
(n - m - 1)a^{-1} + O(1)
\]
as \(a \to 0\), where the bounded term \(O(1)\) depends only on the geometry of \(S\).
\end{lem}

The tubular hypersurfaces \(\Tub{T_{x_\infty} S}{a}\) with varying \(a\) give a foliation of \(T_{x_\infty}\R^n \setminus T_{x_\infty} S\) by cylinders with constant mean curvature in the metric \(\conffact{\infty}^{4/(n - 2)} \delta\). 
In the next lemma we see that this foliation contains exactly one minimal hypersurface.

\begin{lem}\label{lem:mean-curvature-of-cylinders}
Let
\[
\hat a
\coloneqq
\left(D_{n, m} \frac{1 - C_{n, m}}{C_{n, m}}\right)^{1/(n - m - 2)}
\]
where
\[
C_{n, m}
\coloneqq
\frac{(n - 2)(n - m - 1)}{2(n - 1)(n - m - 2)}
\]
and
\[
D_{n, m} \coloneqq \int_{\R^m} (1 + |\eta|^2)^{-(n-2)/2} \, d\eta.
\]
Then the mean curvature of \(\Tub{T_{x_\infty} S}{a}\) in the metric \(\conffact{\infty}^{4/(n - 2)} \delta\) is zero for \(a = \hat a\), negative for \(0 < a < \hat a\) and positive for \(a > \hat a\).
\end{lem}
\begin{proof}
For \(\zeta \in \Tub{T_{x_\infty} S}{a}\) it holds that
\[\begin{split}
\conffact{\infty}(\zeta)
&=
1 + \int_{T_{x_\infty} S} |\zeta - \eta|^{-(n-2)} \, d\eta
\\
&=
1 + \int_{T_{x_\infty} S} (a^2 + |\eta|^2)^{-(n-2)/2} \, d\eta
\\
&=
1 + a^{-(n - m - 2)} \int_{\R^m} (1 + |\eta|^2)^{-(n-2)/2} \, d\eta
\\
&=
1 + a^{-(n - m - 2)} D_{n, m}.
\end{split}\]
Since the mean curvature of \(\Tub{T_{x_\infty} S}{a}\) in the metric \(\delta\) is \((n-m-1)a^{-1}\) we find that its mean curvature in the metric \(\conffact{\infty}^{4/(n - 2)} \delta\) is
\[\begin{split}
&\conffact{\infty}^{-2/(n-2)}\left(\frac{n - m - 1}{a} + 2 \frac{n - 1}{n - 2} \deriv{}{a} \ln \conffact{\infty}\right)\\
&\qquad\qquad= (1 + a^{-\gamma} D_{n, m})^{-\frac{2}{n-2}} \frac{1}{a} \left((\gamma + 1) - 2 \gamma \frac{n - 1}{n - 2} \frac{a^{-\gamma} D_{n, m}}{1 + a^{-\gamma} D_{n, m}}\right)
\end{split}\]
where we have set \(\gamma \coloneqq n - m - 2\).
From this it is not complicated to verify the claim of the lemma. 
\end{proof}

In the following proposition we show that a tubular hypersurface around \(S\) with radius in a certain interval has positive mean curvature in \((\R^n \setminus S, g_\epsilon)\).

\begin{prop}\label{prop:outer-barriers}
There are constants \(\Couter > 0\) and \(\Router > 0\) such that for all sufficiently small \(\epsilon > 0\) it holds for all \(\Couter \epsilon \leq a \leq \Router\) that the tubular hypersurface \(\Tub{S}{a}\) has positive mean curvature in \((\R^n \setminus S, g_\epsilon)\).
\end{prop}
\begin{proof}
Suppose that the result does not hold.
Then for every choice of \(C\) and \(R\) there is a sequence \(\epsilon_k \to 0\) and a sequence \(a_k\) with \(C \epsilon_k \leq a_k \leq R\) such that \(\Tub{S}{a_k}\) contains at least one point \(x_k\) where the mean curvature is nonpositive. 

We first consider the case where \(a_k \to 0\) and there is a subsequence for which \(\epsilon_k/a_k\) converges to a positive number \(L\). 
Note that then \(L \leq 1/C\). 
We rescale by \(La_k\) around the point \(x_k\) using the map \(\transportandrescale{x_\infty}{x_k}{La_k}\).
Then by \prettyref{lem:convergence-of-S} we have
\[
(\transportandrescale{x_\infty}{x_k}{La_k})^{-1}(\Tub{S}{a_k}) \to \Tub{T_{x_\infty} S}{1/L}.
\]
It holds that
\[\begin{split}
&(La_k)^{-2}(\transportandrescale{x_\infty}{x_k}{La_k})^*(g_{\epsilon_k})(\zeta) \\
&\quad =
\left(1 + \epsilon_k^{n-m-2} \int_S |\transportandrescale{x_\infty}{x_k}{La_k}(\zeta) - y|^{-(n-2)} \, dy \right)^{4/(n-2)} \delta \\
&\quad =
\left(1 + \left(\frac{\epsilon_k}{La_k}\right)^{n-m-2} F_k(\zeta) \right)^{4/(n-2)} \delta
\end{split}\]
where
\[
F_k(\zeta)
=
(La_k)^{n-m-2} \int_S |\transportandrescale{x_\infty}{x_k}{La_k}(\zeta) - y|^{-(n-2)} \, dy.
\]
We conclude from \prettyref{prop:uniform-integral-convergence} applied with \(\epsilon_k = L a_k\) that
\[\begin{split}
\lim_{k \to \infty} (La_k)^{-2}(\transportandrescale{x_\infty}{x_k}{La_k})^*(g_{\epsilon_k})(\zeta)
&=
\left(1 + F_\infty(\zeta) \right)^{4/(n-2)} \delta
\\
&=
\left( \conffact{\infty}(\zeta) \right)^{4/(n-2)} \delta
\end{split}\]
where
\[
F_\infty(\zeta)
=
\int_{T_{x_\infty}S} |\zeta - \eta|^{-(n-2)} \, d\eta.
\]
Choose \(C > \hat a\).
Then \(1/L \geq C > \hat a\) so that by \prettyref{lem:mean-curvature-of-cylinders} the limit cylinder \(\Tub{T_{x_\infty} S}{1/L}\) has positive mean curvature in the limit metric \(\conffact{\infty}^{4/(n-2)} \delta\).
This contradicts the assumption that the mean curvature of \(\Tub{S}{a_k}\) in the metric \(g_{\epsilon_k}\) is nonpositive at \(x_k\) for every \(k\).

Second, we consider the case where \(a_k \to 0\) and \(\epsilon_k/a_k \to 0\). 
In this case we rescale by \(a_k\) around the point \(x_k\) using the map \(\transportandrescale{x_\infty}{x_k}{a_k}\), in which case 
\[
(\transportandrescale{x_\infty}{x_k}{a_k})^{-1}(\Tub{S}{a_k}) \to \Tub{T_{x_\infty} S}{1}.
\]
From a computation like in the previous case we find that 
\[ 
\lim_{k \to \infty} a_k^{-2}(\transportandrescale{x_\infty}{x_k}{a_k})^*(g_{\epsilon_k}) \\
=
\delta. 
\]
Since \(\Tub{T_{x_\infty} S}{1}\) has positive mean curvature in the metric \(\delta\) we again find a contradiction.

Finally, we consider the case where there is a subsequence for which \(a_k \to a_\infty \neq 0\).
Then \(a_\infty \leq R\).
If \(R\) is sufficiently small, the tubular hypersurfaces \(\Tub{S}{a_k}\) stay in a compact set bounded away from \(S\) and the metrics \(g_{\epsilon_k}\) tend uniformly to the Euclidean metric \(\delta\) on this compact set. 
Choosing \(R\) sufficiently small compared to the geometry of \(S\) we find from \prettyref{lem:mean-curvature-in-delta} that the limit surface \(\Tub{S}{a_\infty}\) has positive mean curvature in the limit metric \(\delta\), which contradicts the assumption of nonpositive mean curvature of \(\Tub{S}{a_k}\) in \(g_{\epsilon_k}\) at \(x_k\).
\end{proof}

The following proposition states that any tubular hypersurface around \(S\) with sufficiently small radius has negative mean curvature in \((\R^n \setminus S, g_\epsilon)\).
The proof is similar to the proof of the previous proposition, and we omit it.

\begin{prop}\label{prop:inner-barriers}
There is a constant \(\Cinner > 0\) such that for all sufficiently small \(\epsilon > 0\) it holds for all \(0 < a \leq \Cinner \epsilon\) that the tubular hypersurface \(\Tub{S}{a}\) has negative mean curvature in \((\R^n \setminus S, g_\epsilon)\).
\end{prop}

\section{The location of outer area minimizing stationary hypersurfaces}

\label{sec:the-location-of-outer-area-minimizing-stationary-hypersurfaces}

The purpose of this section is to prove that for small \(\epsilon\), all outer area minimizing stationary bounding hypersurfaces in \((\R^n \setminus S, g_\epsilon)\) are contained in the region between \(\Tub{S}{\Cinner \epsilon}\) and \(\Tub{S}{\Couter \epsilon}\).
In this paper, we use the word \enquote{hypersurface} to mean a smooth hypersurface.
Note that this means that the closure of a hypersurface might contain singular points.
We begin by introducing the relevant terminology.

Let \((M, g)\) be an asymptotically Euclidean Riemannian manifold.
A hypersurface \(\Sigma \subset M\) is called \emph{bounding} if its closure \(\overline \Sigma\) is compact and \(\Sigma\) is the reduced boundary of a Caccioppoli set (see for instance \cite{Giusti84}, in particular \cite[Chapter~3 and Remark~3.2]{Giusti84}), the complement of which is a neighborhood of the asymptotically Euclidean end.
We call this Caccioppoli set \emph{the region inside \(\Sigma\)}.
The bounding hypersurface \(\Sigma\) \emph{encloses} a bounding hypersurface \(\Sigma'\) if the region inside \(\Sigma\) contains the region inside \(\Sigma'\).
We say that a bounding hypersurface in \((\R^n \setminus S, g_\epsilon)\) \emph{encloses the end at \(S\)} if it encloses \(\Tub{S}{a}\) for all sufficiently small \(a > 0\).
We define the \emph{regular} and \emph{singular} sets of a hypersurface \(\Sigma\) as follows:
The regular set \(\reg(\Sigma)\) is the set of points in \(\overline \Sigma\) which have a neighborhood in which \(\overline \Sigma\) is a connected hypersurface.
The singular set is \(\sing(\Sigma) = \overline \Sigma \setminus \reg(\Sigma)\).
We say that \(\Sigma\) is \emph{nonsingular} if \(\sing(\Sigma) = \emptyset\).
We only consider hypersurfaces \(\Sigma\) for which \(\Sigma = \reg(\Sigma)\), in other words that no points where \(\Sigma\) is smooth have been left out.

Let \(\Sigma\) be a bounding hypersurface in \((M, g)\), and let \(E\) be the region inside \(\Sigma\).
Then \(\Sigma\) is \emph{outer area minimizing} if its area is not greater than the perimeter of any Caccioppoli set containing \(E\).
A hypersurface \(\Sigma\) in a Riemannian manifold \((M, g)\) is \emph{stationary} if the first variation of its area vanishes when \(\Sigma\) is deformed in the direction of any vector field with compact support in \(M\).
A hypersurface is \emph{stable} if, in addition, the second variation of area is nonnegative for such vector fields.
We will use the concept of stationarity for \(\Sigma\) in \(M = \R^n \setminus S\) as well as in \(M = \R^n\).
Note that every outer area minimizing stationary hypersurface is stable.

\begin{figure}
 \centering{}
 \input{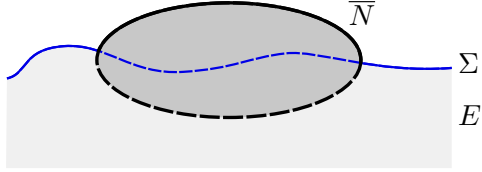}
 \caption{
 	\(\partial(E \cup N)\) is an outward variation of \(\Sigma\) and hence has at least the same area.
	This gives a bound for the area of \(\Sigma \cap N\).
 }
\label{fig:area-bound}
\end{figure}

The outer area minimizing property gives local area bounds, as expressed in the following lemma.
We denote \(k\)-dimensional Hausdorff measure with respect to a metric \(g\) by \(\Hm{g}{k}\).
\begin{lem}\label{lem:local-area-bound}
Let \(\Sigma\) be an outer area minimizing hypersurface in an asymptotically Euclidean Riemannian manifold \((M, g)\).
Let \(N \subset M\) be an open set with compact closure and piecewise smooth boundary.
Then
\[
\Hm{g}{n-1}(\Sigma \cap N)
\leq
\Hm{g}{n-1}(\partial N).
\]
\end{lem}
\begin{proof}
Let \(E\) be the region inside \(\Sigma\).
The set \(\partial(E \cup N)\) is an outward variation of \(\Sigma\), as shown in \prettyref{fig:area-bound}, and hence
\[
\Hm{g}{n-1}(\Sigma)
\leq
\Hm{g}{n-1}(\partial(E \cup N)).
\]
Splitting this inequality into the part inside \(\overline N\) and the part outside \(\overline N\) we have
\[\begin{split}
&\Hm{g}{n-1}\left(\Sigma \cap \overline N\right) + \Hm{g}{n-1}\left(\Sigma \setminus \overline N\right)
\\
&\qquad \leq
\Hm{g}{n-1}\left(\partial(E \cup N) \cap \overline N\right) + \Hm{g}{n-1}\left(\partial(E \cup N) \setminus \overline N\right).
\end{split}\]
Since
\(
\Sigma \setminus \overline N
=
\partial(E \cup N) \setminus \overline N
\)
the above inequality reduces to
\[
\Hm{g}{n-1}(\Sigma \cap \overline N)
\leq
\Hm{g}{n-1}(\partial(E \cup N) \cap \overline N).
\]
Since
\(
\Sigma \cap N
\subseteq
\Sigma \cap \overline N
\)
and
\(
\partial(E \cup N) \cap \overline N
\subseteq
\partial N
\)
we obtain the desired bound.
\end{proof}

For the proof of \prettyref{prop:convergence-without-rescaling} and \prettyref{prop:local-graphicality} we will use a convergence result for stable hypersurfaces.
We have not been able to find a proof of this result in the exact form we need, but it follows from the work of Schoen and Simon in \cite{SchoenSimon81} and we sketch a proof.
\begin{thm}[Schoen--Simon \cite{SchoenSimon81}]\label{thm:minimal-surface-convergence}
Let \(\Omega\) be an orientable smooth manifold of dimension \(n\) and let \(K \subset \Omega\) be compact.
Let \(h_\infty\) be a Riemannian metric on \(\Omega\), let \(h_k\) be a sequence of smooth Riemannian metrics on \(\Omega\), and let \(\Gamma_k\) be a sequence of hypersurfaces in \(\Omega\).
Assume that
\begin{itemize}
	\item \(\overline \Gamma_k \cap K \neq \emptyset\) for all \(k\),
	\item each \(\Gamma_k\) is a stable hypersurface for the metric \(h_k\),
	\item \(\Hm{h_\infty}{n-3}(\sing(\Gamma_k)) = 0\),
	\item for every open set \(\Omega'\) with compact closure \(\overline{\Omega'} \subseteq \Omega\)
	\begin{itemize}
		\item the metrics \(h_k\) converge to \(h_\infty\) uniformly in \(C^r(\Omega')\) for every \(r\), and
		\item \(\limsup_{k \to \infty} \Hm{h_k}{n-1}(\Gamma_k \cap \Omega') < \infty\).
	\end{itemize}
\end{itemize}
Then there is
\begin{itemize}
\item a subsequence of \(\Gamma_k\), still denoted \(\Gamma_k\), and
\item a hypersurface \(\Gamma_\infty \subset \Omega\)
\end{itemize}
such that
\begin{itemize}
	\item \(\Gamma_k \to \Gamma_\infty\) as varifolds,
	\item \(\Hm{h_\infty}{\alpha}(\sing(\Gamma_\infty)) = 0\) if \(\alpha > n - 8\) and \(\alpha \geq 0\),
	\item \(\Gamma_\infty \cap K \neq \emptyset\),
	\item \(\Gamma_\infty\) is a stable hypersurface in the metric \(h_\infty\),
	\item for every open set \(\Omega'\) with compact closure \(\overline{\Omega'} \subseteq \Omega \setminus \sing(\Gamma_\infty)\)
	\begin{itemize}
		\item \(\Hm{h_\infty}{n-1}(\Gamma_\infty \cap \Omega') \leq \limsup_{k \to \infty} \Hm{h_k}{n-1}(\Gamma_k \cap \Omega')\), and
		\item the subsequence \(\Gamma_k\) converges smoothly to \(\Gamma_\infty\) (disregarding multiplicity) on \(\Omega'\).
	\end{itemize}
\end{itemize}
\end{thm}
The smooth convergence of \(\Gamma_k\) to \(\Gamma_\infty\) on \(\Omega'\), disregarding multiplicity, means that if \(k\) is sufficiently large then in normal coordinates for \(\Gamma_\infty\), each \(\Gamma_k\) is the union of graphs of a finite number of smooth functions on \(\Gamma_\infty\), and these functions converge in \(C^\infty(\Gamma_\infty \cap \Omega')\) to \(0\) as \(k \to \infty\).
\begin{proof}
In this proof, we let \(\sing_{C^2}(\Gamma)\) denote the singular part of \(\Gamma\) in the \(C^2\) sense.
Cover \(\Omega\) by balls of a fixed small radius.
Applying \cite[Theorem~2]{SchoenSimon81} in each ball we obtain a (possibly empty) stable \(C^2\) hypersurface \(\Gamma_\infty\), with \(\Hm{h_\infty}{\alpha}(\sing_{C^2}(\Gamma_\infty)) = 0\) if \(\alpha > n - 8\) and \(\alpha \geq 0\), to which a subsequence \(\Gamma_k\) converges in the sense of varifolds.
Moreover, there is at least one ball in which \(\Gamma_\infty \cap K \neq \emptyset\).
Taking a diagonal subsequence we can patch these locally defined hypersurfaces together to a hypersurface \(\Gamma_\infty \subset \Omega\) with \(\Gamma_\infty \cap K \neq \emptyset\).

Let \(\Omega'\) be an open set with compact closure \(\overline{\Omega'} \subseteq \Omega \setminus \sing_{C^2}(\Gamma_\infty)\).
Since \(\Gamma_k \to \Gamma_\infty\) as varifolds, we have \(\Hm{h_\infty}{n-1}(\Gamma_\infty \cap \Omega') \leq \limsup_{k \to \infty} \Hm{h_k}{n-1}(\Gamma_k \cap \Omega')\).

Using \cite[Inequality~(1.18)]{SchoenSimon81} one can prove that varifold convergence implies convergence in Hausdorff distance (see also \cite[Remark~10 on p.~779]{SchoenSimon81}).
Hence \cite[Theorem~1]{SchoenSimon81} is applicable at every point of \(\Gamma_\infty \cap \Omega'\), telling us that the hypersurfaces \(\Gamma_k \cap \Omega'\) eventually become graphs of functions over the tangent space of \(\Gamma_\infty \cap \Omega'\) at the chosen point.
This theorem also gives uniform \(C^2\) bounds for these functions, which implies subsequential convergence in \(C^{1,\alpha}\) for all \(\alpha \in (0, 1)\).
The hypersurfaces \(\Gamma_k \cap \Omega'\) are graphs of functions satisfying the minimal surface equations for a converging sequence of metrics.
By ellipticity of the minimal surface equations, we conclude \(C^\infty\) convergence and that \(\Gamma_\infty \cap \Omega'\) is smooth, which implies that \(\sing(\Gamma_\infty) = \sing_{C^2}(\Gamma_\infty)\).
\end{proof}

Having completed these preliminaries, we continue the study of the location of outer area minimizing stationary bounding hypersurfaces in \((\R^n \setminus S, g_\epsilon)\).
Since the metrics \(g_\epsilon\) converge uniformly to the Euclidean metric outside of a neighborhood of \(S\) as \(\epsilon \to 0\) we conclude the following.
\begin{prop}\label{prop:large-coordinate-sphere}
Let \(\sphere^{n-1}(R)\) denote the sphere of radius \(R\) around the origin in \(\R^n\).
There is a constant \(\Rend\) such that, for small \(\epsilon\), it holds that \(\sphere^{n-1}(R)\) has positive mean curvature in \((\R^n, g_\epsilon)\) if \(R \geq \Rend\).
\end{prop}

Using the maximum principle of Solomon and White \cite{SolomonWhite89} we can prove the following proposition.

\begin{prop}\label{prop:minsurf-a-priori-bound}
Every stationary bounding hypersurface \(\Sigma\) in \((\R^n \setminus S, g_\epsilon)\) is contained in the region between \(\Tub{S}{\Cinner \epsilon}\) and \(\sphere^{n-1}(\Rend)\).
\end{prop}
\begin{proof}
Let \(a \coloneqq \inf_{x \in \Sigma} \dist^\delta(x, S)\).
It holds that \(0 < a\) since \(\overline \Sigma \subset \R^n \setminus S\) is compact.
Let \(N\) denote the region outside \(\Tub{S}{a}\).
Suppose for contradiction that \(a \leq \Cinner \epsilon\).
By \prettyref{prop:inner-barriers} the hypersurface \(\partial N = \Tub{S}{a}\) has negative mean curvature.
Since \(\Sigma\) is a stationary hypersurface with \(\overline \Sigma \cap \partial N \neq \emptyset\), \solomonwhite{} implies that the mean curvature of at least one connected component of \(\partial N\) is zero, which is a contradiction.
Hence \(\Sigma\) is contained in the region outside \(\Tub{S}{\Cinner \epsilon}\).

A similar argument using \(a \coloneqq \sup_{x \in \Sigma} \dist^\delta(x, 0)\) and \prettyref{prop:large-coordinate-sphere} proves that \(\Sigma\) is contained in the region inside \(\sphere^{n-1}(\Rend)\).
\end{proof}

The previous result restricts the location of outer area minimizing stationary hypersurfaces to within a large coordinate sphere.
We will now prove that they actually collapse to \(S\) as \(\epsilon \to 0\).
The proof is inspired by the proof of \cite[Theorem~3.2]{ChruscielMazzeo03}.

\begin{prop}\label{prop:convergence-without-rescaling}
Let \(R > 0\).
For sufficiently small \(\epsilon > 0\), every outer area minimizing stationary hypersurface \(\Sigma\) in \((\R^n \setminus S, g_\epsilon)\) with \(\HmE{n-3}(\sing(\Sigma)) = 0\) is enclosed by \(\Tub{S}{R}\).
\end{prop}
\begin{proof}
Suppose for contradiction that there is a sequence \(\epsilon_k \to 0\), a sequence \(\Sigma_k\) of outer area minimizing stationary hypersurfaces in \((\R^n \setminus S, g_{\epsilon_k})\), and a sequence of points \(x_k \in \overline \Sigma_k\) such that \(\dist^\delta(x_k, S) > R\).
Note that \(g_{\epsilon_k}\) converges to the Euclidean metric \(\delta\) uniformly on compact subsets in \(\R^n \setminus S\).
Moreover, since \(\Sigma_k\) is outer area minimizing we have \(\Hm{g_{\epsilon_k}}{n-1}(\Sigma_k) \leq \Hm{g_{\epsilon_k}}{n-1}(\sphere^{n-1}(\Rend))\) which is uniformly bounded.
By \prettyref{thm:minimal-surface-convergence}, applied with \(\Omega\) equal to \(\R^n \setminus S\) and \(K\) being the closure of the region between \(\Tub{S}{R}\) and \(\sphere^{n-1}(\Rend)\), there is then a nonempty stable hypersurface \(\Sigma_\infty\) in \((\R^n \setminus S, \delta)\) such that \(\Hm{\delta}{\alpha}(\sing(\Sigma_\infty) \setminus S) = 0\) for all nonnegative \(\alpha > n - 8\) and
\[
\Hm{\delta}{n-1}(\Sigma_\infty \cap \Omega') \leq \limsup_{k \to \infty} \Hm{g_{\epsilon_k}}{n-1}(\Sigma_k \cap \Omega')
\]
for all compact sets \(\Omega' \subset \R^n \setminus S\).
Moreover, \(\Sigma_\infty\) is enclosed by \(S^{n-1}(\Rend)\) since this holds for each \(\Sigma_k\) by \prettyref{prop:minsurf-a-priori-bound}.

The hypersurface \(\Sigma_\infty\) is stationary in \((\R^n \setminus S, \delta)\).
We now prove that it is stationary as a hypersurface in \((\R^n, \delta)\).
We need to prove that
\[
\int_{\Sigma_\infty} \div^\delta_{\Sigma_\infty}(X) \, d\Hm{\delta}{n-1} = 0
\]
for all compactly supported vector fields \(X\) on \(\R^n\).
Here \(\div^\delta_{\Sigma_\infty}(X)\) is the divergence of \(X\) along the hypersurface \(\Sigma_\infty\) in the Euclidean metric \(\delta\).
For this, we define a family of smooth cut-off functions \(\eta_\rho\) for all sufficiently small \(\rho > 0\) by
\[
\eta_\rho(x)
\coloneqq
\eta\left( \frac{\dist^\delta(x, S)}{\rho} \right)
\]
where \(\eta \colon [0, \infty) \to [0, 1]\) is a smooth increasing function such that \(\eta(1) = 0\) and \(\eta(2) = 1\), with derivative \(\eta'(t) < 2\) for all \(t\).
Since \(\Sigma_k \to \Sigma_\infty\), and each \(\Sigma_k\) is outer area minimizing so that \prettyref{lem:local-area-bound} is applicable, there are constants \(C_1\), \(C_2\), and \(C_3\) such that
\[\begin{split}
\int_{\Sigma_\infty} |d\eta_\rho| \, d\Hm{\delta}{n-1}
&\leq
\frac{2}{\rho} \Hm{\delta}{n-1}(\Sigma_\infty \cap \operatorname{supp}(d\eta_\rho))
\\
& =
\frac{2}{\rho} \Hm{\delta}{n-1}(\Sigma_\infty \cap \{x \in \R^n \mid \rho < \dist^\delta(x, S) < 2\rho\})
\\
& \leq
\frac{C_1}{\rho} \Hm{g_{\epsilon_k}}{n-1}(\Sigma_k \cap \{x \in \R^n \mid \rho < \dist^\delta(x, S) < 2\rho\})
\\
& \leq
\frac{C_1}{\rho} \Hm{g_{\epsilon_k}}{n-1}(\Tub{S}{\rho} \cup \Tub{S}{2\rho})
\\
& \leq
\frac{C_2}{\rho} \Hm{\delta}{n-1}(\Tub{S}{\rho} \cup \Tub{S}{2\rho})
\\
& =
C_3 \rho^{n-m-2},
\end{split}\]
where the intermediate inequalities hold for all sufficiently large \(k\).
Since \(\Sigma_\infty\) is stationary as a hypersurface in \((\R^n \setminus S, \delta)\) it holds that
\[
\int_{\Sigma_\infty} \div^\delta_{\Sigma_\infty}(\eta_\rho X) \, d\Hm{\delta}{n-1}
=
0.\]
Let \(\pi\) denote the \(\delta\)-orthogonal projection of \(T\R^n\) onto \(T\Sigma_\infty\).
Since \(\div^\delta_{\Sigma_\infty}(\eta_\rho X) = \eta_\rho \div^\delta_{\Sigma_\infty}(X) + \pi(X)(\eta_\rho)\) and \(n - m - 2 \geq 1\) we have
\[\begin{split}
\left| \int_{\Sigma_\infty} \div^\delta_{\Sigma_\infty}(X) \, d\Hm{\delta}{n-1} \right|
&=
\lim_{\rho \to 0} \left| \int_{\Sigma_\infty} \eta_\rho \div^\delta_{\Sigma_\infty}(X) \, d\Hm{\delta}{n-1} \right|
\\
&=
\lim_{\rho \to 0} \left|\int_{\Sigma_\infty} \pi(X)(\eta_\rho) \, d\Hm{\delta}{n-1}\right|
\\
&\leq
\left(\sup_{\Sigma_\infty} |X|\right) \lim_{\rho \to 0} \int_{\Sigma_\infty} |d\eta_\rho| \, d\Hm{\delta}{n-1}
\\
&\leq
\left(\sup_{\Sigma_\infty} |X|\right) \lim_{\rho \to 0} C_3 \rho^{n - m - 2}
\\
&=
0.
\end{split}\]
Hence \(\Sigma\) is stationary as a hypersurface in \((\R^n, \delta)\).

Let \(L \colon \R^n \to \R\) be a nonzero linear function.
Let \(a \coloneqq \sup_{x \in \Sigma} L(x)\).
Since \(\overline \Sigma\) is compact, \(a < \infty\).
The hyperplane \(L^{-1}(a)\) is a connected minimal hypersurface in \((\R^n, \delta)\) and \(\Sigma\) is stationary.
According to \solomonwhite, the hyperplane \(L^{-1}(a)\) is contained in \(\overline \Sigma\) which contradicts compactness of \(\overline \Sigma\).
\end{proof}

We summarize the results of this section in the following proposition.

\begin{prop}\label{prop:location-of-stationary-hypersurfaces}
For all sufficiently small \(\epsilon > 0\), every outer area minimizing stationary hypersurface \(\Sigma\) in \((\R^n \setminus S, g_\epsilon)\) with \(\HmE{n-3}(\sing(\Sigma)) = 0\) must be contained in the region between \(\Tub{S}{\Cinner \epsilon}\) and \(\Tub{S}{\Couter \epsilon}\).
\end{prop}
\begin{proof}
By \prettyref{prop:minsurf-a-priori-bound} and \prettyref{prop:convergence-without-rescaling} it is sufficient to prove that \(\Sigma\) does not intersect the region between \(\Tub{S}{\Couter \epsilon}\) and \(\Tub{S}{\Router}\).
\prettyref{prop:outer-barriers} tells us that this region is foliated by hypersurfaces with positive mean curvature, which means that \(\Sigma\) cannot intersect it without contradicting the maximum principle of Solomon and White, as in \prettyref{prop:minsurf-a-priori-bound}.
\end{proof}

\section{Tubularity of outer area minimizing stationary hypersurfaces}

\label{sec:tubularity-of-outer-area-minimizing-stationary-hypersurfaces}

In this section we will prove \prettyref{prop:global-graphicality}, which states that if \(\epsilon > 0\) is sufficiently small, then every outer area minimizing stationary hypersurface \(\Sigma\) in \((\R^n \setminus S, g_\epsilon)\) which encloses the end at \(S\) and has \(\HmE{n-3}(\sing(\Sigma)) = 0\) is diffeomorphic to the unit normal bundle of \(S\).
The idea of the proof is the following.
By using a contradiction argument it is sufficient to study a sequence of such hypersurfaces \(\Sigma_k\) close to \(x_k\) for a convergent sequence \((x_k, \epsilon_k) \to (x_\infty, 0)\).
We prove that, after rescaling around the points \(x_k\), the hypersurfaces \(\Sigma_k\) converge smoothly to a cylinder around \(T_{x_\infty} S\).
For this convergence we use the area bounds from \prettyref{lem:local-area-bound}, which we obtained by using the fact that the hypersurfaces are outer area minimizing.
Since the rescaled hypersurfaces converge smoothly to a cylinder, they must be the union of graphs of a number of smooth functions on the cylinder provided \(k\) is sufficiently large.
By using the outer area minimizing property again we can prove that, near \(x_k\), each hypersurface is the graph of a single smooth function, thereby proving the proposition.

The contradiction part of the argument is contained in \prettyref{prop:global-graphicality}, while the local result is the following proposition.

\begin{prop}\label{prop:local-graphicality}
Let \(x_k \to x_\infty\) be a convergent sequence of points in \(S\) and let \(\epsilon_k \to 0\).
For each \(k\), let \(\Sigma_k\) be an outer area minimizing stationary hypersurface in \((\R^n \setminus S, g_{\epsilon_k})\) which encloses the end at \(S\) and has \(\HmE{n-3}(\sing(\Sigma_k)) = 0\).
Then for all sufficiently large \(k\) there is an open set \(U_k \subseteq S\) containing \(x_k\) such that, in normal coordinates, the set \(\overline \Sigma_k \cap \exp^\delta((\Cinner \epsilon_k, \Couter \epsilon_k) UNU_k)\) is the graph of a smooth function on the unit normal bundle \(UNU_k\) of \(U_k\).
\end{prop}

\begin{proof}
\setcounter{stepnumber}{0}
Throughout the proof, \(\exp^\delta\) denotes the normal exponential map of \(S \subset \R^n\) in the Euclidean metric \(\delta\), and we consider only \(\epsilon > 0\) which are sufficiently small for \(\exp^\delta \colon [\Cinner \epsilon, \Couter \epsilon] UNS \to \R^n\) to be an embedding.

\stepheader{\step:\label{step:rescaling} Rescaling}
Consider the metrics
\[
h_k \coloneqq \epsilon_k^{-2} \left(\transportandrescale{x_\infty}{x_k}{\epsilon_k}\right)^* (g_{\epsilon_k})
\]
on
\[
\Omega \coloneqq T_{x_\infty}\R^n \setminus T_{x_\infty}S.
\]
By \prettyref{cor:metric-convergence} these metrics converge in \(C^r(\Omega')\) for every \(r \geq 0\) and every open set \(\Omega'\) with compact closure \(\overline{\Omega'} \subseteq \Omega\) to
\[
h_\infty
\coloneqq
\conffact{\infty}^{4/(n - 2)} \delta
\]
where
\[
\conffact{\infty}(\zeta)
=
1 + \int_{T_{x_\infty}S} |\zeta - \eta|^{-(n - 2)} \, d\eta.
\]
Define
\[
\Gamma_k \coloneqq \left(\transportandrescale{x_\infty}{x_k}{\epsilon_k}\right)^{-1}(\Sigma_k).
\]
Then \(\Gamma_k\) is an outer area minimizing stable hypersurface for the metric \(h_k\).

\stepheader{\step:\label{step:convergence} Convergence}
We are now going to apply \prettyref{thm:minimal-surface-convergence} with \(\Omega\), \(h_k\), \(h_\infty\), and \(\Gamma_k\) as defined in \prettyref{step:rescaling}, and with \(K \coloneqq [\Cinner/2, 2\Couter] UN_{x_\infty}S\).
\prettyref{prop:location-of-stationary-hypersurfaces} tells us that the intersection \(\Gamma_k \cap K\) is nonempty for sufficiently large \(k\).
By \prettyref{lem:local-area-bound} it holds that
\[
\Hm{h_k}{n-1}(\Gamma_k \cap \Omega')
\leq
\Hm{h_k}{n-1}(\partial \Omega')
\]
for every open set \(\Omega'\) with compact closure \(\overline{\Omega'} \subseteq \Omega\) and piecewise smooth boundary.
This, together with the smooth convergence \(h_k \to h_\infty\), tells us that
\[
\limsup_{k \to \infty} \Hm{h_k}{n-1}(\Gamma_k \cap \Omega')
\leq
\limsup_{k \to \infty} \Hm{h_k}{n-1}(\partial \Omega')
=
\Hm{h_\infty}{n-1}(\partial \Omega').
\]
With these choices, \prettyref{thm:minimal-surface-convergence} is applicable, giving a stable hypersurface \(\Gamma_\infty\) and a subsequence of \(\Gamma_k\), still denoted \(\Gamma_k\), which are such that
\begin{itemize}
	\item \(\Gamma_k \to \Gamma_\infty\) as varifolds,
	\item \(\Hm{h_\infty}{\alpha}(\sing(\Gamma_\infty)) = 0\) if \(\alpha > n - 8\) and \(\alpha \geq 0\),
	\item \(\Gamma_\infty \cap K \neq \emptyset\),
	\item \(\Gamma_\infty\) is a stable hypersurface in the metric \(h_\infty\),
	\item for every open set \(\Omega'\) with compact closure \(\overline{\Omega'} \subseteq \Omega \setminus \sing(\Gamma_\infty)\)
	\begin{itemize}
		\item \(\Hm{h_\infty}{n-1}(\Gamma_\infty \cap \Omega') \leq \limsup_{k \to \infty} \Hm{h_k}{n-1}(\Gamma_k \cap \Omega')\), and
		\item the subsequence \(\Gamma_k\) converges smoothly to \(\Gamma_\infty\) (disregarding multiplicity) on \(\Omega'\).
	\end{itemize}
\end{itemize}
By the above,
\begin{equation}\label{eq:limit-area-bound}
\Hm{h_\infty}{n-1}(\Gamma_\infty \cap \Omega')
\leq
\Hm{h_\infty}{n-1}(\partial \Omega').
\end{equation}

\stepheader{\step: Identifying the limit}\label{step:identifying-the-limit}
We will now show that \(\Gamma_\infty = \Tub{T_{x_\infty} S}{\hat a}\) with \(\hat a\) as in \prettyref{lem:mean-curvature-of-cylinders}.
Suppose for contradiction that \(\Gamma_\infty \setminus \Tub{T_{x_\infty} S}{\hat a} \neq \emptyset\).
Then either \(\Gamma_\infty\) contains points strictly outside of \(\Tub{T_{x_\infty} S}{\hat a}\), or it contains points strictly inside of \(\Tub{T_{x_\infty} S}{\hat a}\).
We derive a contradiction only in the former case; the argument for a contradiction in the latter case is analogous.
Let
\[
a
\coloneqq
\sup_{\xi \in \Gamma_\infty} \dist(\xi, T_{x_\infty} S)
>
\hat a
\]
as illustrated in \prettyref{fig:Gamma-infty}.
We will use a sequence of translations and an application of \prettyref{thm:minimal-surface-convergence} to reduce to the case where the supremum is attained.

Let \(\xi_i \in \Gamma_\infty\) be a sequence of points such that \(\dist(\xi_i, T_{x_\infty} S) \to a\).
As above, let \(K \coloneqq [\Cinner/2, 2\Couter] UN_{x_\infty}S\).
Let \(\Omega_K\) be an open neighborhood of \(K\) with compact closure \(\overline{\Omega_K} \subset \Omega\) and piecewise smooth boundary.
Note that the metric \(h_\infty\) has translational symmetry along \(T_{x_\infty} S\).
For each \(i\) let \(\tau_i\) be a translation along \(T_{x_\infty} S\) with \(\tau_i(\xi_i) \in K\).
Let \(\Gamma_\infty^i \coloneqq \tau_i(\Gamma_\infty) \cap \Omega_K\).
Then, by the translation symmetry of \(h_\infty\) and by inequality \eqref{eq:limit-area-bound},
\[\begin{split}
\Hm{h_\infty}{n-1}(\Gamma_\infty^i)
&=
\Hm{h_\infty}{n-1}(\tau_i(\Gamma_\infty) \cap \Omega_K)
\\
&=
\Hm{h_\infty}{n-1}(\Gamma_\infty \cap \tau_i^{-1}(\Omega_K))
\\
&\leq
\Hm{h_\infty}{n-1}(\partial(\tau_i^{-1}(\Omega_K)))
\\
&=
\Hm{h_\infty}{n-1}(\partial \Omega_K).
\end{split}\]
This means that
\[
\Hm{h_\infty}{n-1}(\Gamma_\infty^i \cap \Omega')
\leq
\Hm{h_\infty}{n-1}(\partial \Omega_K)
\]
for every open set \(\Omega'\) with compact closure \(\overline{\Omega'} \subseteq \Omega_K\).
The hypersurfaces \(\Gamma_\infty^i\) all intersect the compact set \(K\).
By \prettyref{thm:minimal-surface-convergence}, a subsequence of \(\Gamma_\infty^i\) converges to a stable hypersurface \(\Gamma_\infty^\infty\).

\begin{figure}
\centering
\begin{minipage}{.5\textwidth}
  \centering
  \input{fig9.tex}
  \captionof{figure}{The surface \(\Gamma_\infty\) approaches \(\Tub{T_{x_\infty} S}{a}\) from the inside.}
  \label{fig:Gamma-infty}
\end{minipage}%
\begin{minipage}{.5\textwidth}
  \centering
  \input{fig8.tex}
  \captionof{figure}{After translations, the surface \(\Gamma_\infty^\infty\) is tangent to \(\Tub{T_{x_\infty} S}{a}\) from the inside.}
  \label{fig:Gamma-infty-infty}
\end{minipage}
\end{figure}

Let \(\xi_\infty\) be a limit point of the sequence \(\tau_i(\xi_i)\).
Since \(\Gamma_\infty^i \to \Gamma_\infty^\infty\) in Hausdorff distance (see \cite[Remark~10 on p.~779]{SchoenSimon81}) it holds that \(\xi_\infty \in \overline{\Gamma^\infty_\infty}\).
This proves that the supremum \(a\) of \(\dist(\cdot, T_{x_\infty} S)\) on \(\overline{\Gamma_\infty^\infty}\) is attained in an interior point \(\xi_\infty\).
The stationary hypersurface \(\Gamma_\infty^\infty \subset \Omega_K\) is contained in the region inside \(\Tub{T_{x_\infty} S}{a}\), and \(\xi_\infty \in \overline{\Gamma_\infty^\infty} \cap \Tub{T_{x_\infty} S}{a}\) as in \prettyref{fig:Gamma-infty-infty}.
Since we have assumed for contradiction that \(a > \hat a\), \prettyref{lem:mean-curvature-of-cylinders} tells us that \(\Tub{T_{x_\infty} S}{a}\) has positive mean curvature.
As in \prettyref{prop:minsurf-a-priori-bound}, this contradicts the maximum principle of Solomon and White.
Hence \( \sup_{\xi \in \Gamma_\infty} \dist(\xi, T_{x_\infty} S) = \hat a\).
Similarly, \(\inf_{\xi \in \Gamma_\infty} \dist(\xi, T_{x_\infty} S) = \hat a\), so that \({\Gamma_\infty \subseteq \Tub{T_{x_\infty} S}{\hat a}}\).

The hypersurfaces \(\Sigma_k\) enclose the end at \(S\), so each \(\overline \Gamma_k\) separates \(\left(\transportandrescale{x_\infty}{x_k}{\epsilon_k}\right)^{-1}(S)\) from faraway points in \(\R^n\), in the sense that every continuous curve from \(\left(\transportandrescale{x_\infty}{x_k}{\epsilon_k}\right)^{-1}(S)\) from faraway points in \(\R^n\) intersects \(\overline \Gamma_k\).
This means that the hypersurface \(\overline \Gamma_\infty\) separates \(T_{x_\infty} S\) from faraway points in the normal space \(N_{x_\infty} S\).
The only subset of \(\Tub{T_{x_\infty} S}{\hat a}\) which separates \(T_{x_\infty} S\) from such points is \(\Tub{T_{x_\infty} S}{\hat a}\) itself, so \(\overline \Gamma_\infty = \Tub{T_{x_\infty} S}{\hat a}\).
By our convention, \(\Gamma_\infty = \reg \Gamma_\infty\), so it holds that \(\Gamma_\infty\) is nonsingular.

\stepheader{\step: Determining the multiplicity}
To simplify notation we introduce a \emph{rescaling in the normal direction} defined for \(r \in \R\) and \(\zeta \in T_{x_\infty} \R^n\) by
\[
r * \zeta = \pi(\zeta) + r(\zeta - \pi(\zeta)).
\]
Recall that \(\pi\) denotes the orthogonal projection \(\pi \colon T_{x_\infty}\R^n \to T_{x_\infty}S\).

In the previous step we found that the hypersurfaces \(\Gamma_k\) converge smoothly to \(\Tub{T_{x_\infty} S}{\hat a}\) on each compact set, if we disregard multiplicities.
We will now prove that the multiplicity is one.
Choose an open neighborhood \(U \subset T_{x_\infty} S\) of \(0\) with compact closure.
The smooth convergence of \(\Gamma_k\) to \(\Tub{T_{x_\infty} S}{\hat a}\) on \((\Cinner, \Couter) * \Tub{U}{1}\) implies for any \(\rho > 0\) that for all sufficiently large \(k\), there is a finite set of smooth functions
\[
\Psi^1_k, \ldots, \Psi^N_k \colon \Tub{U}{1} \to (\Cinner, \Couter)
\]
such that
\[
\sup_{\substack{\zeta \in \Tub{U}{1}\\i \in \{1, \ldots, N\}}} |\Psi^i_k(\zeta) - \hat a|
<
\rho
\]
and such that the maps \(\omega \mapsto \Psi^i_k(\omega) * \omega\) together give a surjective map
\[
\bigsqcup_{i \in \{1, \ldots, N\}} \Tub{U}{1}
\to
\Gamma_k \cap (\Cinner, \Couter) * \Tub{U}{1},
\]
see \prettyref{fig:Delta-k}.
The functions \((\Psi^i_k)_{1 \leq i \leq N}\) may agree at some points, but since we assume that \(\HmE{n-3}(\sing(\Sigma_k)) = 0\), this intersection set has zero \((n-3)\)-dimensional Hausdorff measure.
Thus, we may order the functions so that \(\Psi^1_k < \Psi^2_k < \cdots < \Psi^N_k\) except on a set with zero \((n-3)\)-dimensional Hausdorff measure.

Our goal is now to show that \(N = 1\).
To do this, we will use an argument similar to the one we used in \prettyref{lem:local-area-bound} to obtain area bounds.
Since \(\Gamma_k\) is bounding, it encloses a region \(E_k\).
Choose an open neighborhood \(V \subset T_{x_\infty} S\) of \(0\) with smooth boundary such that \(\overline V \subset U\).
Let
\[
\outwardvariation{k}
\coloneqq
\{r * \omega \mid \omega \in \Tub{V}{1}, \Psi^1_k(\omega) \leq r \leq \Psi^{N - 1}_k(\omega)\},
\]
as shown in \prettyref{fig:Xi-k}.
The boundary of \(\outwardvariation{k}\) has a \enquote{vertical} component
\[
\partial^\perp \outwardvariation{k}
\coloneqq
\{r * \omega \mid \omega \in \partial \Tub{V}{1}, \Psi^1_k(\omega) \leq r \leq \Psi^{N - 1}_k(\omega)\}.
\]
Of course \(E_k \subseteq E_k \cup \outwardvariation{k}\) so \(\partial(E_k \cup \outwardvariation{k})\) is an outward variation of \(\Gamma_k\).
\begin{figure} 
\centering
\begin{minipage}{.5\textwidth}
  \centering
  \input{fig10.tex}
  \captionof{figure}{The hypersurface \(\Gamma_k\) is the union of graphs of functions \(\Psi^i_k\). The set \(E_k\) is shown in light grey.}
  \label{fig:Delta-k}
\end{minipage}%
\begin{minipage}{.5\textwidth}
  \centering
  \input{fig11.tex}
  \captionof{figure}{The set \(\outwardvariation{k}\) is shown in dark grey.}
  \label{fig:Xi-k}
\end{minipage}
\end{figure}
In this variation the boundary is changed by adding the hypersurface \(\partial (E_k \cup \outwardvariation{k}) \setminus \partial E_k\) and removing the hypersurface \(\partial E_k \setminus \partial (E_k \cup \outwardvariation{k})\).
To estimate the area added and removed, we note that
\[
\partial (E_k \cup \outwardvariation{k}) \setminus \partial E_k
\subseteq
\partial^\perp \outwardvariation{k}
\]
and
\[
\partial E_k \setminus \partial (E_k \cup \outwardvariation{k})
=
\bigcup_{i = 1}^{N - 1} \{\Psi^i_k(\omega) * \omega \mid \omega \in \Tub{V}{1}\},
\]
up to sets of zero \((n-3)\)-dimensional Hausdorff measure.
On the set \((\Cinner, \Couter) * \Tub{U}{1}\) all metrics \(h_k\) and \(h_\infty\) are conformal to the Euclidean metric with uniformly bounded conformal factors.
Therefore there are constants \(C^+\) and \(C^-\) depending on \(U\), \(\Cinner\), and \(\Couter\), but not depending on \(\rho\) and \(k\), such that
\[
\Hm{h_k}{n-1}(\partial^\perp \outwardvariation{k})
\leq
\rho C^+ \Hm{h_\infty}{n-2}(\partial \Tub{V}{1})
\]
and
\[
\Hm{h_k}{n-1}\left(\{\Psi^i_k(\omega) * \omega \mid \omega \in \Tub{V}{1}\}\right)
\geq
C^- \Hm{h_\infty}{n-1}(\Tub{V}{1})
\]
for all \(1 \leq i \leq N - 1\).

The hypersurface \(\Gamma_k\) is outer area minimizing, since \(\Sigma_k\) is outer area minimizing by assumption.
Hence \(\partial(E_k \cup \outwardvariation{k})\) has at least the same area as \(\Gamma_k\) in \(h_k\).
This means that
\[
\Hm{h_k}{n-1}(\partial E_k \setminus \partial (E_k \cup \outwardvariation{k}))
\leq
\Hm{h_k}{n-1}(\partial (E_k \cup \outwardvariation{k}) \setminus \partial E_k).
\]
With the above estimates we get
\[
N - 1
\leq
\frac{
	\rho C^+ \Hm{h_\infty}{n-2}(\partial \Tub{V}{1})
}{
	C^- \Hm{h_\infty}{n-1}(\Tub{V}{1})
}.
\]
Choosing \(\rho\) such that the right hand side is less than \(1\) we conclude that \(N = 1\).
We have now proved that, for all sufficiently large \(k\), there is a single function \(\Psi_k \coloneqq \Psi_k^1 \colon \Tub{U}{1} \to (\Cinner, \Couter)\) such that the map \(\omega \mapsto \Psi_k(\omega) * \omega\) is a diffeomorphism 
\[
\Tub{U}{1}
\to
\Gamma_k \cap (\Cinner, \Couter) * \Tub{U}{1}.
\]

We will now use the function \(\Phi_k\) to write \(\Sigma_k\) as a graph.
Let \(P_k \coloneqq \transportandrescale{x_\infty}{x_k}{\epsilon_k}(T_{x_\infty} S)\) of \(\R^n\).
This is the tangent space of \(S\) at \(x_\infty\) considered as an affine subspace of \(\R^n\).
Define the map
\[\widehat\Psi_k \colon UNP_k \supset UN\widehat U_k \to \Tub{P_k}{\epsilon_k} \to \Tub{T_{x_\infty} S}{1} \to (\Cinner \epsilon_k, \Couter \epsilon_k)\]
by
\[\widehat\Psi_k(\omega) \coloneqq \epsilon_k \Phi_k\left(\left(\transportandrescale{x_\infty}{x_k}{\epsilon_k}\right)^{-1}\left(\exp^\delta(\epsilon_k \omega)\right)\right)\]
where \(\widehat U_k \coloneqq \transportandrescale{x_\infty}{x_k}{\epsilon_k}(U) \subset P_k\).
Let
\[
\widehat \Omega_k
\coloneqq
\exp^\delta((\Cinner \epsilon_k, \Couter \epsilon_k) UN\widehat U_k).
\]
With these choices, \(\overline \Sigma_k\) is the graph of \(\widehat\Psi_k\) over \(\widehat U_k\) in the sense that
\[
\overline \Sigma_k \cap \widehat \Omega_k
=
\exp^\delta\left(\{\widehat\Psi_k(\omega) \omega \mid \omega \in UN\widehat U_k\}\right).
\]
If \(|d\widehat\Psi_k|\) is bounded and there is an open set \(\widetilde \Omega_k \subseteq \widehat \Omega_k\) containing \(x_k\) in which the tangent spaces to \(S\) are sufficiently close to being parallel to \(P_k\), then \(\overline \Sigma_k\) is the graph of a smooth function \(\widetilde \Psi_k\) over some open set \(\widetilde U_k \subset S\) containing \(x_k\) in the sense that
\[
\overline \Sigma_k \cap \widetilde \Omega_k
=
\exp^\delta\left(\{\widetilde \Psi_k(\omega) \omega \mid \omega \in UN\widetilde U_k\}\right),
\]
as shown in \prettyref{fig:local-graphicality}.
\begin{figure}
 \centering{}
 \input{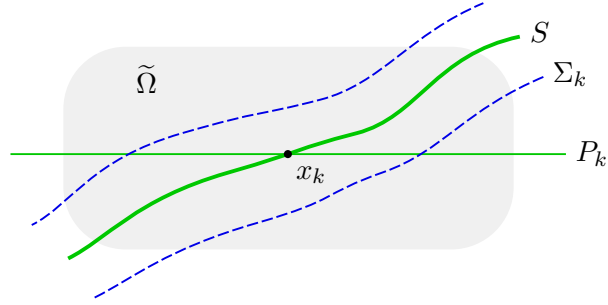}
 \caption{
  Since \(\Sigma_k\) is the graph of a function on \(UNP_k\) with bounded gradient and \(S\) is close to \(P_k\), it follows that it is also the graph of a function on \(UN\widetilde U_k\) for some \(\widetilde U_k \subset S\).
 }
\label{fig:local-graphicality}
\end{figure}
By letting \(k\) be sufficiently large, and letting \(\widetilde \Omega_k \subseteq \widehat \Omega_k\) be sufficiently small, this condition on the tangent spaces of \(S\) is satisfied.
Moreover, \(|d\widehat\Psi_k|\) is bounded since \(|d\Psi_k| \to 0 \) when \(k \to 0\).
Hence, in normal coordinates, the set \(UN\widetilde U_k \cap \overline \Sigma_k\) is the graph of a smooth function on the unit normal bundle \(UN\widetilde U_k\) of \(\widetilde U_k\).
Choosing \(U_k = \widetilde U_k\) completes the proof.
\end{proof}

\prettyref{prop:local-graphicality} globalizes in the following manner.
\begin{prop}\label{prop:global-graphicality}
For sufficiently small \(\epsilon > 0\) it holds that every outer area minimizing stationary hypersurface \(\Sigma\) in \((\R^n \setminus S, g_\epsilon)\) which encloses the end at \(S\) and has \(\HmE{n-3}(\sing(\Sigma)) = 0\) is diffeomorphic to the unit normal bundle \(UNS\).
In fact, each such hypersurface \(\Sigma\) is the graph of a smooth function on \(UNS\) in normal coordinates for \(S\).
\end{prop}
\begin{proof}
Suppose for contradiction that this is not true.
Then there is a sequence \(\epsilon_k \to 0\) and outer area minimizing stationary hypersurfaces \(\Sigma_k\) in \((\R^n \setminus S, g_{\epsilon_k})\) such that \(\Sigma_k\) is not the graph of a smooth function on \(UNS\) in normal coordinates for \(S\).
From \prettyref{prop:location-of-stationary-hypersurfaces} we know that \(\overline \Sigma\) is contained in the image of \((\Cinner \epsilon_k, \Couter \epsilon_k) UNS\) in normal coordinates.
Hence, for each \(k\) there is a point \(x_k \in S\) such that if \(U_k \subseteq S\) is an open neighborhood of \(x_k\) then \(\overline \Sigma_k \cap \exp^\delta((\Cinner \epsilon_k, \Couter \epsilon_k) UNU_k)\) is not the graph of a smooth function on \(UNU_k\) in normal coordinates.
Taking a convergent subsequence \(x_k \to x_\infty\), this contradicts \prettyref{prop:local-graphicality}.
\end{proof}

\section{Proof of the main theorem}

\label{sec:proof-of-the-main-theorem}

In this section we prove existence and uniqueness of an \emph{outermost} outer area minimizing stationary hypersurface \(\Sigma_\epsilon\) which encloses the end at \(S\) and has \(\HmE{n-3}(\sing(\Sigma_\epsilon)) = 0\), in other words one which is not enclosed by any other such hypersurface.
Further, we prove that this hypersurface is the outermost apparent horizon, which by definition is the boundary of the trapped region.
This, combined with the conclusion of the previous section proves our main theorem.

There are general results in dimensions \(3 \leq n \leq 7\) stating that the boundary of the trapped region is smooth, and hence is the unique outermost outer area minimizing stationary hypersurface; see \cite[Theorem~5.1]{Eichmair10} and \cite[Theorem~4.6]{AnderssonEichmairMetzger11}.
In our special case, the uniqueness can be deduced by a simpler argument which does not need the dimensional restriction.

First, we prove that there are outer area minimizing stationary hypersurfaces outside any obstacle in the form of a bounding hypersurface with negative mean curvature.

\begin{prop}\label{prop:outer-minimizing-existence}
Let \(T\) be a bounding nonsingular hypersurface in \((\R^n \setminus S, g_\epsilon)\) which encloses the end at \(S\) and has negative mean curvature.
There is at least one outer area minimizing stationary hypersurface \(\Sigma\) in \((\R^n \setminus S, g_\epsilon)\) which encloses \(T\) and is such that \(\Hm{h_\infty}{\alpha}(\sing(\Sigma)) = 0\) if \(\alpha > n - 8\) and \(\alpha \geq 0\).
\end{prop}
\begin{proof}
Let \(L \subset \R^n\) be the closure of the region inside \(T\).
Let \(\Omega \subset \R^n\) be the open region between \(T\) and \(\sphere^{n-1}(\Rend)\).
Consider the Caccioppoli sets \(F\) such that \(L \subseteq F \subseteq L \cup \Omega\).
By the existence theorem for minimal sets (see for instance \cite[Theorem~1.20]{Giusti84}) there is a set \(E\) which minimizes perimeter among all such sets.
This minimizer \(E\) may be chosen such that \(\partial E = \overline{\partial^* E}\); see \cite[Theorem~4.4]{Giusti84}.
Here \(\partial^* E\) is the reduced boundary of \(E\); see \cite[Definition 3.3]{Giusti84}.

The reduced boundary \(\partial^* E\) is a rectifiable varifold (see \cite[Theorem~14.3]{Simon}) which is area minimizing with respect to \(T\) and \(S(\Rend)\) as obstacles, since \(E\) minimizes perimeter with respect to such variations.
By \prettyref{prop:large-coordinate-sphere} the sphere \(\sphere^{n-1}(\Rend)\) has positive mean curvature, and by hypothesis \(T\) has negative mean curvature.
Hence it follows from \solomonwhite{} that \(\overline{\partial^*E}\) does not intersect \(\partial \Omega\).
This means that \(\partial^* E\) is actually a solution to an area minimization problem without obstacles, and hence smooth (compare \cite[Theorem~8.4]{Giusti84}).
Let \(\Sigma = \partial^* E\).
Now \(\sing \Sigma \subseteq \partial E \setminus \partial^* E\), and it holds by \cite[Theorem~11.8]{Giusti84} that \(\HmE{\alpha}(\sing(\partial E \setminus \partial^* E)) = 0\) for all \(\alpha > n - 8\) and \(\alpha \geq 0\).

By construction, \(\Sigma\) encloses \(T\).
It coincides with the global area minimizer outside of \(T\) since the global area minimizer cannot intersect the region outside \(\sphere^{n-1}(\Rend)\) (which by \prettyref{prop:large-coordinate-sphere} can be foliated by positive mean curvature spheres) without contradicting the maximum principle of Solomon and White.
Hence \(\Sigma\) is also outer area minimizing.
\end{proof}

By \prettyref{prop:inner-barriers}, the tubular hypersurface \(\Tub{S}{\Cinner \epsilon}\) has negative mean curvature if \(\epsilon > 0\) is sufficiently small, which yields the following.
\begin{cor}\label{cor:outer-minimizing-existence}
For all sufficiently small \(\epsilon > 0\) there is at least one outer area minimizing stationary hypersurface \(\Sigma\) in \((\R^n \setminus S, g_\epsilon)\) which encloses the end at \(S\) and is such that \(\Hm{h_\infty}{\alpha}(\sing(\Sigma)) = 0\) if \(\alpha > n - 8\) and \(\alpha \geq 0\).
\end{cor}

Note that it follows from \prettyref{prop:global-graphicality} that the hypersurface \(\Sigma\) from \prettyref{cor:outer-minimizing-existence} is nonsingular.

To prove that there is a unique outermost outer area minimizing stationary hypersurface, we also need to be able to find such hypersurfaces outside the union of two obstacles with possibly nonempty singular parts.
However, in this case, we only need to consider obstacles which are outer area minimizing and stationary.

\begin{prop}\label{prop:pairwise-outer-existence}
Let \(\Sigma_1\) and \(\Sigma_2\) be outer area minimizing stationary hypersurfaces in \((\R^n \setminus S, g_\epsilon)\) which enclose the end at \(S\) and have \(\HmE{n-3}(\sing(\Sigma_1)) = \HmE{n-3}(\sing(\Sigma_2)) = 0\).
If \(\epsilon > 0\) is sufficiently small, there is an outer area minimizing stationary hypersurface \(\Sigma\) which encloses both \(\Sigma_1\) and \(\Sigma_2\), and is such that \(\Hm{h_\infty}{\alpha}(\sing(\Sigma)) = 0\) if \(\alpha > n - 8\) and \(\alpha \geq 0\).
\end{prop}
\begin{proof}
Assume without loss of generality that \(\Sigma_1\) and \(\Sigma_2\) have no common connected components.

It follows from \prettyref{prop:global-graphicality} that \(\Sigma_1\) and \(\Sigma_2\) are nonsingular.
Let \(E_1\) denote the closure of the region inside \(\Sigma_1\) and let \(E_2\) denote the closure of the region inside \(\Sigma_2\).
Let \(L \coloneqq E_1 \cup E_2 \subset \R^n\).
Let \(\Omega \subset \R^n\) be the open region between \(\partial L\) and \(\sphere^{n-1}(\Rend)\).
As in \prettyref{prop:outer-minimizing-existence} we obtain a Caccioppoli set \(E\) with \(\partial E = \overline{\partial^* E}\) which minimizes perimeter among those which contain \(E_1 \cup E_2\).
Let \(\Sigma = \partial^* E\).
By the maximum principle of Solomon and White, \(\overline \Sigma\) cannot intersect the outer boundary \(\sphere^{n-1}(\Rend)\).

We begin by proving that \(\overline \Sigma\) cannot intersect \(\Sigma_1 \cap \Sigma_2\).
Suppose for contradiction that there is a point \(x \in \overline \Sigma \cap \Sigma_1 \cap \Sigma_2\).
Let \(\nu_i\) denote the outward-directed unit normal vector field of \(\Sigma_i\).
It cannot hold that \(\nu_1 = -\nu_2\) anywhere, since \prettyref{prop:global-graphicality} tells us that \(\Sigma_1\) and \(\Sigma_2\) are graphs over \(UNS\) in normal coordinates.
Let \((\nu_*)_x\) be the unit vector in the direction \((\nu_1)_x + (\nu_2)_x\), so that \(g_\epsilon((\nu_*)_x, (\nu_i)_x) > 0\).

For the proof we need, in a neighborhood of \(x\), a stationary nonsingular hypersurface \(\Sigma_*\) with \(x \in \Sigma_* \subset \overline{E_1 \cup E_2}\), and normal vector \((\nu_*)_x\) at \(x\).
If \(\nu_1 = \nu_2\) at \(x\), we choose \(\Sigma_*\) to be an open neighborhood of \(x\) in \(\Sigma_1\).
If not, then the intersection of \(\Sigma_1\) and \(\Sigma_2\) is transverse, in which case \(I \coloneqq \Sigma_1 \cap \Sigma_2\) is a smooth submanifold of codimension 2 in \(\R^n\).
Extend \(\nu_*\) be the unit vector field on \(I\) in direction \(\nu_1 + \nu_2\).
Let \(\Sigma_*\) be a hypersurface which contains \(I\), is orthogonal to \(\nu_*\) along \(I\), and has nonpositive mean curvature.
This exists since the mean curvature of \(I\) in direction \(\nu_*\) can be compensated by the curvature in the direction orthogonal to \(I\) and \(\nu_*\).
After shrinking \(\Sigma_*\) if necessary it holds that \(\Sigma_* \subset \overline{E_1 \cup E_2}\).
Extend \(\nu_*\) to a unit normal vector field on all of \(\Sigma_*\).

If \(\Sigma_* \subseteq \overline \Sigma\), it holds that \(\Sigma_*\), \(\Sigma_1\) and \(\Sigma_2\) all agree in a neighborhood of \(x\) since \(\Sigma\) encloses \(\Sigma_1\) and \(\Sigma_2\), which in turn enclose \(\Sigma_*\).
This proves that \(\overline \Sigma \cap \Sigma_1 \cap \Sigma_2\) is an open subset of both \(\Sigma_1\) and \(\Sigma_2\).
It is also a closed subset since \(\overline \Sigma\), \(\Sigma_1\) and \(\Sigma_2\) are closed sets.
Since we assumed that \(\Sigma_1\) and \(\Sigma_2\) have no common connected component, this is a contradiction.

Suppose instead that \(\Sigma_*\) is not a subset of \(\overline \Sigma\).
The proof of \cite[Theorem,~p.~686]{SolomonWhite89} then gives a vector field \(v\) which defines a variation which strictly decreases the area of \(\overline \Sigma\).
It is not obvious that this vector field is outward-directed along \(\Sigma_1\) and \(\Sigma_2\).
However, by choosing \(\epsilon\) and \(s\) in \cite[pp.~687-690]{SolomonWhite89} sufficiently small, we may arrange that the functions \(u_{s, t, \epsilon}\) in that proof are \(C^1\)-close to the function \(u\) for all \(t\) sufficiently close to \(0\).
This means that \(\tau\) may be chosen arbitrarily close to \(0\), so that the function \(u_{s, \tau}\) in the proof is arbitrarily close to \(u\) in \(C^1\) norm.
This, in turn, means that the vector field \(v\) can be chosen arbitrarily close to our vector field \(\nu_*\), so that \(v\) is outward-directed along \(\Sigma_1\) and \(\Sigma_2\) in a neighborhood of \(x\).
Since \(\Sigma\) minimizes area outside of \(E_1 \cup E_2\) and the variation along \(v\) decreases area, we have a contradiction.
This completes the proof that \(\overline \Sigma \cap \Sigma_1 \cap \Sigma_2 = \emptyset\).

Let \(x \in \overline \Sigma \cap \Sigma_1\).
Then \(x \notin \Sigma_2\), so we can apply the maximum principle of Solomon and White to conclude that \(x\) is an interior point of \(\overline \Sigma \cap \Sigma_1\) in \(\Sigma_1\).
The set \(\overline \Sigma \cap \Sigma_1\) is closed in \(\Sigma_1\).
Hence every connected component of \(\Sigma_1\) is either contained in \(\overline \Sigma\) or disjoint from \(\overline \Sigma\).
Similarly, every connected component of \(\Sigma_2\) is either contained in \(\overline \Sigma\) or disjoint from \(\overline \Sigma\).

Fix \(i \in \{1, 2\}\).
We will show, using the argument from \cite[Theorem~4]{White10}, that every connected component of \(\overline \Sigma\) either coincides with a connected component of \(\Sigma_i\) or is disjoint from \(\Sigma_i\).
Suppose for contradiction that there is a point \(x \in \Sigma_i \cap (\overline{\overline \Sigma \setminus \Sigma_i})\).
Let \(W\) be the connected component of \(\Sigma_1\) containing \(x\).
Let \(W' = \Sigma - W\), where we view \(\Sigma\) and \(W\) as unit density rectifiable varifolds.
Since \(W\) is stationary and \(\Sigma\) minimizes area to first order in the complement of \(E_1 \cup E_2\), it holds that \(W'\) minimizes area to first order in this region.
We can then apply the maximum principle of Solomon and White in a neighborhood of \(x\) to see that the support of \(W'\) includes \(W\), which is a contradiction.
Hence every connected component of \(\overline \Sigma\) either coincides with a connected component of \(\Sigma_i\) or is disjoint from \(\Sigma_i\).

The connected components which coincide with connected components of \(\Sigma_1\) or \(\Sigma_2\) are stationary, since the hypersurfaces \(\Sigma_1\) and \(\Sigma_2\) are stationary.
The connected components which are disjoint from \(\Sigma_1\) and \(\Sigma_2\) are solutions to an area minimization problem without obstacles, and hence stationary.
This means that \(\Sigma\) is stationary.
Since it is the global area minimizer in the complement of \(E_1 \cup E_2\) it is also outer area minimizing.
As in \prettyref{prop:outer-minimizing-existence} it follows that \(\Hm{h_\infty}{\alpha}(\sing(\Sigma)) = 0\) if \(\alpha > n - 8\) and \(\alpha \geq 0\).
\end{proof}

Having proved that there is an outer area minimizing stationary hypersurface enclosing the end at \(S\), we turn our attention to proving that there is an outermost such hypersurface, and that this outermost hypersurface is unique.
\begin{prop}\label{prop:outermost-uniqueness}
For all sufficiently small \(\epsilon > 0\), there is a unique outermost outer area minimizing stationary hypersurface \(\Sigma_\epsilon\) in \((\R^n \setminus S, g_\epsilon)\) which encloses the end at \(S\) and has \(\HmE{n-3}(\sing(\Sigma_\epsilon)) = 0\).
\end{prop}
\begin{proof}
We begin by proving that there is at least one outermost outer area minimizing stationary hypersurface.
The set of outer area minimizing stationary hypersurfaces \(\Sigma\) which enclose the end at \(S\) and have \(\HmE{n-3}(\sing(\Sigma)) = 0\) is nonempty by \prettyref{cor:outer-minimizing-existence} and forms a partially ordered set under the relation that \(\Sigma_1 \leq \Sigma_2\) if \(\Sigma_2\) encloses \(\Sigma_1\).
We want to prove that there is a maximal element under this partial order, since such an element is outermost.
By Zorn's lemma it is sufficient to verify that every nonempty chain has an upper bound.
Let \(A\) be a chain.
For \(\Sigma \in A\), let \(\Omega_\Sigma\) be the interior of the region inside \(\Sigma\).
Then \(\bigcup_{\Sigma \in A} \Omega_\Sigma\) is an open cover of itself, and since it is a subset of \(\R^n\) it has a countable subcover \(\bigcup_{i = 1}^\infty \Omega_{\Sigma_i}\).
The sequence \(\Omega_{\Sigma_1} \subseteq \Omega_{\Sigma_2} \subseteq \dots\) is increasing since \(A\) is a chain.
By \prettyref{thm:minimal-surface-convergence} there is a subsequence of \(\Sigma_k\) converging as varifolds to a stationary hypersurface \(\Sigma_\infty\) with \(\Hm{g_\epsilon}{\alpha}(\sing(\Sigma_\infty)) = 0\) if \(\alpha > n - 8\) and \(\alpha \geq 0\).
Then \(\Hm{g_\epsilon}{n-1}(\Sigma_\infty) = \lim_{k \to \infty} \Hm{g_\epsilon}{n-1}(\Sigma_k)\), so since each \(\Sigma_k\) is outer area minimizing it holds that \(\Sigma_\infty\) is outer area minimizing.
Hence \(\Sigma\) is an upper bound for the chain.
Hence Zorn's lemma tells us that there is an outermost outer area minimizing stationary hypersurface \(\Sigma\) which encloses the end at \(S\) and has \(\HmE{n-3}(\sing(\Sigma)) = 0\).

Suppose that there were two different such hypersurfaces \(\Sigma_1\) and \(\Sigma_2\).
Then the previous proposition produces one which encloses both of them, which is a contradiction.
Hence there is a unique outermost outer area minimizing stationary hypersurface \(\Sigma_\epsilon\) which encloses the end at \(S\) and has \(\HmE{n-3}(\sing(\Sigma_\epsilon)) = 0\).
\end{proof}

The uniqueness of \(\Sigma_\epsilon\) implies that it coincides with the outermost apparent horizon, as shown in the following proposition.
\begin{prop}\label{prop:outermost-apparent-horizon}
If \(\epsilon > 0\) is sufficiently small, the hypersurface \(\Sigma_\epsilon\) in \((\R^n \setminus S, g_\epsilon)\) is the boundary of the trapped region.
\end{prop}
\begin{proof}
We know from \prettyref{prop:global-graphicality} that \(\Sigma_\epsilon\) is nonsingular and hence contained in the trapped region \(\trappedregion_\epsilon\) of \((\R^n \setminus S, g_\epsilon)\).

Let \(x \in \partial \trappedregion_\epsilon\).
We want to show that \(x \in \Sigma_\epsilon\).
Choose a sequence \(x_k \to x\) contained in the interior of the trapped region.
Then for each \(k\) there is a weakly outer trapped surface \(T_k\) enclosing the point \(x_k\).
Since \(T_k\) is weakly outer trapped, it also encloses the end at \(S\).
Let \(\Sigma_k\) denote an outer area minimizing stationary hypersurface enclosing \(T_k\) obtained from \prettyref{prop:outer-minimizing-existence}.
Using \prettyref{thm:minimal-surface-convergence} as in the proof of \prettyref{prop:outermost-uniqueness} we obtain a subsequential limit hypersurface \(\Sigma'\) which is still outer area minimizing, encloses the end at \(S\), and encloses or contains \(x\).
Since \(\Sigma_\epsilon\) is an outermost outer area minimizing stationary hypersurface, it encloses \(\Sigma'\) by \prettyref{prop:pairwise-outer-existence}.
Hence \(x\) is enclosed by or contained in \(\Sigma_\epsilon\).
This completes the proof.
\end{proof}

Finally, we prove the main theorem of this paper.
\begin{proof}[Proof of \prettyref{thm:main-theorem}]
There is a unique outermost outer area minimizing stationary hypersurface \(\Sigma_\epsilon\) with \(\HmE{n-3}(\sing(\Sigma_\epsilon)) = 0\) by \prettyref{prop:outermost-uniqueness}, and this hypersurface is a smooth graph over \(UNS\) in normal coordinates by \prettyref{prop:global-graphicality}.
From \prettyref{prop:outermost-apparent-horizon} we know that \(\Sigma_\epsilon\) is the outermost apparent horizon of \((\R^n \setminus S, g_\epsilon)\).
\end{proof}

\bibliographystyle{abbrv}
\bibliography{references}

\end{document}